\documentclass[final, onefignum, onetabnum]{siamart220329}

\usepackage{lipsum}
\usepackage{amsfonts}
\usepackage{graphicx}
\usepackage{epstopdf}
\usepackage{algorithmic}
\ifpdf
  \DeclareGraphicsExtensions{.eps,.pdf,.png,.jpg}
\else
  \DeclareGraphicsExtensions{.eps}
\fi
\usepackage{bm} 
\usepackage{mathrsfs} 
\usepackage{stmaryrd}
\usepackage{amssymb}
\usepackage{tikz}

\definecolor{colorOne}{RGB}{0, 39, 77}
\definecolor{colorTwo}{RGB}{97, 0, 57}
\definecolor{colorThree}{RGB}{176, 1, 54}

\newcommand{\timeVariable}{t} 
\newcommand{\spaceVariable}{x} 
\newcommand{\vectorial}[1]{\bm{#1}} 
\newcommand{\conservedVariable}{u} 
\newcommand{\indexDirection}{k} 
\newcommand{\spatialDimensionality}{d} 
\newcommand{\flux}{\varphi} 
\newcommand{\numberVelocities}{q} 
\newcommand{\numberLinks}{W} 
\newcommand{\idEst}{\emph{i.e.}}
\newcommand{\confer}{\emph{cf.}}
\newcommand{\exempliGratia}{\emph{e.g.}}
\newcommand{\spaceStep}{\Delta \spaceVariable} 
\newcommand{\timeStep}{\Delta \timeVariable} 
\newcommand{\latticeVelocity}{\lambda} 
\newcommand{\discreteVelocityLetter}{c} 
\newcommand{\indexVelocity}{i} 
\newcommand{\indexLink}{\ell} 
\newcommand{\relatives}{\mathbb{Z}} 
\newcommand{\naturals}{\mathbb{N}} 
\newcommand{\naturalsWithoutZero}{\mathbb{N}^*} 
\newcommand{\integerInterval}[2]{\llbracket#1, #2\rrbracket} 
\newcommand{\discrete}[1]{\mathsf{#1}} 
\newcommand{\distributionFunctionLetter}{f}
\newcommand{\distributionFunction}{\discrete{\distributionFunctionLetter}} 
\newcommand{\indexSpace}{j} 
\newcommand{\indexTime}{n} 
\newcommand{\symmetricDistributionFunctionLetter}{s}
\newcommand{\symmetricDistributionFunction}{\discrete{\symmetricDistributionFunctionLetter}}
\newcommand{\antiSymmetricDistributionFunctionLetter}{a}
\newcommand{\antiSymmetricDistributionFunction}{\discrete{\antiSymmetricDistributionFunctionLetter}}
\newcommand{\definitionEquality}{:=} 
\newcommand{\collided}{\star} 

\newcommand{\relaxationParameter}{\omega} 
\newcommand{\relaxationParameterSymmetric}{\omega_{\discrete{s}}} 
\newcommand{\relaxationParameterAntiSymmetric}{\omega_{\discrete{a}}} 
\newcommand{\relaxationTimeSymmetric}{\epsilon_{\discrete{s}}}
\newcommand{\relaxationTimeAntiSymmetric}{\epsilon_{\discrete{a}}}
\newcommand{\atEquilibrium}{\textnormal{eq}} 
\newcommand{\conservedVariableDiscrete}{\discrete{\conservedVariable}} 
\newcommand{\equilibriumCoefficientLinear}{\mathscr{L}} 
\newcommand{\equilibriumCoefficientFlux}{\mathscr{N}} 
\newcommand{\differential}{\textnormal{d}} 
\newcommand{\maximumInitialDatum}{u_{\infty}} 
\newcommand{\minimumDistribution}{\underline{f}} 
\newcommand{\invariantCompactSetDistributions}{K} 
\newcommand{\lbmScheme}[2]{$\textnormal{D}_{#1}\textnormal{Q}_{#2}$} 
\newcommand{\reals}{\mathbb{R}} 
\newcommand{\initial}{\circ} 
\newcommand{\lebesgueSpace}[1]{L^{#1}} 
\newcommand{\boundedVariationSpace}{\textnormal{BV}} 
\newcommand{\totalVariation}[1]{\textnormal{TV}(#1)}
\newcommand{\totalVariationVectorial}[1]{\textnormal{TV}(#1)}
\newcommand{\continuousSpace}[1]{C^{#1}} 
\newcommand{\timeGridPoint}[1]{\timeVariable^{#1}} 
\newcommand{\spaceGridPoint}[1]{\vectorial{\spaceVariable}_{#1}} 
\newcommand{\collisionOperator}{\mathcal{R}} 
\newcommand{\cell}[1]{C_{#1}} 
\newcommand{\discreteMark}{\Delta} 
\newcommand{\distributionFunctionsAsFunction}{\vectorial{\distributionFunctionLetter}_{\discreteMark}} 
\newcommand{\distributionFunctionsAsFunctionComponent}[1]{{\distributionFunctionLetter}_{\discreteMark, #1}} 
\newcommand{\distributionFunctionsAsFunctionWithStep}[1]{\vectorial{\distributionFunctionLetter}_{#1}} 
\newcommand{\conservedVariableDiscreteAsFunction}{\conservedVariable_{\discreteMark}} 
\newcommand{\conservedVariableDiscreteAsFunctionWithStep}[1]{\conservedVariable_{#1}} 
\newcommand{\distributionFunctionsAsFunctionSpace}[1]{\vectorial{\distributionFunctionLetter}_{\discreteMark}^{#1}} 
\newcommand{\distributionFunctionsAsFunctionSpaceAnyLetter}[2]{\vectorial{#1}_{\discreteMark}^{#2}}
\newcommand{\conservedVariableDiscreteAsFunctionSpace}[1]{\conservedVariable_{\discreteMark}^{#1}} 
\newcommand{\transpose}[1]{#1^{\textsf{T}}} 
\usepackage{dsfont}
\newcommand{\indicatorFunction}[1]{\mathds{1}_{#1}} 
\newcommand{\lebesgueInSpace}[1]{L_{\vectorial{\spaceVariable}}^{#1}}
\newcommand{\lebesgueTime}[1]{L_{\timeVariable}^{#1}} 
\newcommand{\petitLebesgueSpace}[1]{\ell^{#1}} 
\newcommand{\besovSpace}[3]{B_{#1, #3}^{#2}}

\newcommand{\strong}[1]{\emph{#1}} 
\newcommand{\bigO}[1]{\mathcal{O}(#1)}
\newcommand{\canonicalBasisVector}[1]{\vectorial{e}_{#1}}
\newcommand{\limitConservedMoment}{\overline{\conservedVariable}}
\newcommand{\limitDistributionFunction}{\overline{\distributionFunctionLetter}}
\newcommand{\finalTime}{T}
\newcommand{\testFunction}{\psi}
\newcommand{\testFunctionWithStep}{\psi_{\discreteMark}}
\newcommand{\testFunctionDiscrete}{\discrete{\psi}}

\newcommand{\monotoneZoneTwoD}{\mathcal{M}}
\newcommand{\monotoneSegmentBGK}{\mathcal{M}_{\textnormal{BGK}}}
\newcommand{\kineticEntropy}{s}
\newcommand{\krushkovParameter}{\kappa}
\newcommand{\sign}{\text{sgn}}

\DeclareMathOperator*{\argmin}{argmin}

\def\Xint#1{\mathchoice
   {\XXint\displaystyle\textstyle{#1}}%
   {\XXint\textstyle\scriptstyle{#1}}%
   {\XXint\scriptstyle\scriptscriptstyle{#1}}%
   {\XXint\scriptscriptstyle\scriptscriptstyle{#1}}%
   \!\int}
\def\XXint#1#2#3{{\setbox0=\hbox{$#1{#2#3}{\int}$}
     \vcenter{\hbox{$#2#3$}}\kern-.5\wd0}}

\def\dashint{\Xint-}

\usepackage{xparse}

\NewDocumentCommand{\refNoColor}{O{red}mo}{%
  \begingroup
  \hypersetup{hidelinks=true}%
  \ref{#2}%
  \IfValueT{#3}{%
    \color{#1}{#3}%
  }%
  \endgroup
}

\newsiamremark{remark}{Remark}
\newsiamremark{hypothesis}{Hypothesis}
\crefname{hypothesis}{Hypothesis}{Hypotheses}
\newsiamthm{claim}{Claim}

\headers{Monotonicity and convergence of TRT LBM schemes}{D. Aregba-Driollet, and T. Bellotti}

\title{Monotonicity and convergence of two-relaxation-times lattice Boltzmann schemes for a non-linear conservation law}

\author{Denise Aregba-Driollet\thanks{Université de Bordeaux, CNRS, Bordeaux INP, IMB, UMR 5251, 33400 Talence, France.}
\and Thomas Bellotti\thanks{Université Paris-Saclay, CNRS, CentraleSupélec, Laboratoire EM2C \& Fédération de Mathématiques de CentraleSupélec, 91190, Gif-sur-Yvette, France.}}

\usepackage{amsopn}

\ifpdf
\hypersetup{
  pdftitle={Monotonicity and convergence of two-relaxation-times lattice Boltzmann schemes for a non-linear conservation law},
  pdfauthor={D. Aregba-Driollet, and T. Bellotti}
}
\fi

\allowdisplaybreaks 

\begin{document}


\maketitle

\begin{abstract}
    We address the convergence analysis of lattice Boltzmann methods for scalar non-linear conservation laws, focusing on two-relaxation-times (TRT) schemes. Unlike Finite Difference/Finite Volume methods, lattice Boltzmann schemes offer exceptional computational efficiency and parallelization capabilities. However, their monotonicity and $L^{\infty}$-stability remain underexplored. Extending existing results on simpler BGK schemes, we derive conditions ensuring that TRT schemes are monotone and stable by leveraging their unique relaxation structure. Our analysis culminates in proving convergence of the numerical solution to the weak entropy solution of the conservation law. Compared to BGK schemes, TRT schemes achieve reduced numerical diffusion while retaining provable convergence. Numerical experiments validate and illustrate the theoretical findings.
\end{abstract}
    
    

\section{Introduction}

We consider the following \strong{scalar non-linear conservation law}, endowed with an initial condition, as studied in \cite{bellotti2023monotonicity, aregba2024convergence}: 
\begin{equation}\label{eq:conservationLaw}
    \begin{cases}
        \partial_{\timeVariable}\conservedVariable(\timeVariable, \vectorial{\spaceVariable}) + \sum\limits_{\indexDirection = 1}^{\spatialDimensionality} \partial_{\spaceVariable_{\indexDirection}}\flux_{\indexDirection}(\conservedVariable(\timeVariable, \vectorial{\spaceVariable})) = 0, \qquad \timeVariable>0, \quad &\vectorial{\spaceVariable} \in \reals^{\spatialDimensionality}, \\
        \conservedVariable(0, \vectorial{\spaceVariable}) = \conservedVariable^{\initial}(\vectorial{\spaceVariable}), \qquad &\vectorial{\spaceVariable} \in \reals^{\spatialDimensionality}.
    \end{cases}
\end{equation}
Here, we assume that the fluxes $\flux_{\indexDirection}$ and initial data $\conservedVariable^{\initial}$ fulfill the usual assumptions: $\flux_{\indexDirection} \in \continuousSpace{1}(\reals)$ and $\conservedVariable^{\initial} \in \lebesgueSpace{1}(\reals^{\spatialDimensionality}) \cap \lebesgueSpace{\infty}(\reals^{\spatialDimensionality}) \cap \boundedVariationSpace(\reals^{\spatialDimensionality})$.
Moreover, we focus on non-constant fluxes $\flux_{\indexDirection}$ satisfying $\flux_{\indexDirection}(0) = 0$.

The convergence analysis of numerical methods for \eqref{eq:conservationLaw} relies on the notion of \strong{monotonicity} \cite[Chapter 3]{godlewski1991hyperbolic}, closely tied to $\lebesgueSpace{\infty}$-stability.
While monotonicity has been extensively studied for Finite Difference/Finite Volume schemes, its exploration for \strong{lattice Boltzmann schemes} remains limited. Lattice Boltzmann schemes are particularly attractive due to their exceptional computational efficiency and parallelization potential, which stem from their intrinsic algorithmic structure. However, discussions on monotonicity and $\lebesgueSpace{\infty}$-stability for these schemes are scarce in the literature.

\subsection{State of the art and current-research standpoint}

The study of monotonicity for lattice Boltzmann schemes has primarily focused on the one-dimensional setting. For a simple scheme with two unknowns, $\lebesgueSpace{\infty}$-stability has been characterized for the linear case in \cite{dellacherie2014construction}. In the non-linear setting, monotonicity and related properties were investigated under different relaxation regimes: under-relaxation in \cite{caetano2024result} and over-relaxation in \cite{bellotti2023monotonicity}. However, the technique of proof of each of these works does not work for the full range of parameters (namely, relaxation parameter and Courant number) where such properties hold, which is the reason why the former considers only under-relaxation, and the latter only over-relaxation.
The approach by \cite{caetano2024result} relies on the original lattice Boltzmann scheme and the fact that the relaxation is a convex combination. Conversely, \cite{bellotti2023monotonicity} transforms the lattice Boltzmann scheme into its corresponding multi-step Finite Difference scheme, and analyzes monotonicity on this latter, generalizing the usual notion of monotone scheme \cite[Definition 3.1]{godlewski1991hyperbolic}.

A significant step towards generality was achieved in \cite{aregba2024convergence}, which analyzed single-relaxation-time (SRT, or BGK, standing for Bhatnagar-Gross-Krook) multi-di\-men\-sio\-nal lattice Boltzmann schemes. This approach relies on the well-known correspondence between lattice Boltzmann schemes and discrete kinetic BGK systems \cite{graille2014approximation}. However, the method requires equilibria to be monotone non-decreasing with respect to their sole argument, thus can hardly handle MRT (multiple-relaxation-times) schemes.
However, this paper inspires our way of proceeding by observing that the monotonicity of the equilibria, plus a condition on the relaxation parameter, imply that the \strong{relaxation operator} is monotone non-decreasing with respect to each of its arguments.

In the present paper, we determine conditions under which---for the \strong{TRT (two-relaxation-times) schemes} \cite{ginzburg2008two} under scrutiny---the relaxation operator is monotone non-de\-creas\-ing, which was the intuition which \cite{dubois2020notion} relied on.
Thereby, we adapt the well-known proof by \cite{crandall1980monotone}, and show \strong{convergence} of the numerical solution to the unique weak entropy solution of \eqref{eq:conservationLaw}.
We do not pursue the idea of \cite{bellotti2023monotonicity}: recast to Finite Difference schemes---for two reasons. 
First, it would not easily provide the entire set of parameters where monotonicity can be expected, since this approach does not allow the study of the under-relaxation regime even for a simple scheme with two unknowns.
Second, it is indeed difficult to precisely describe the coefficients of the characteristic polynomial of the matrix representing generic schemes, except in a particular case recalled in what follows.

\subsection{The interest of TRT schemes}

TRT schemes offer a balance between the simplicity of BGK schemes and the generality of MRT schemes \cite{d1992generalized}.
While the transport step of any lattice Boltzmann scheme is very simple, \idEst{} diagonal in the space of the distribution functions, relaxation may or may not adhere to this structure.
BGK relaxations do, for they are diagonal in the space of the distribution functions, and the only coupling originates from the conserved moment, through equilibria.
Conversely, with MRT schemes, the relaxation phase is diagonal in another basis, under which the transport step is hard to analyze.
With TRT schemes, the relaxation is diagonal in a relatively simple ($2\times2$-block-diagonal) basis and $2\times2$-block-diagonal in the space of the distribution functions.

The reason why one would be interested in looking at TRT instead of simpler BGK schemes is the following.
TRT schemes introduce an additional degree of freedom through a second relaxation parameter. 
Specifically, we see in what follows that TRT schemes can achieve higher accuracy by increasing one relaxation parameter (which decreases numerical diffusion) \strong{without sacrificing monotonicity} (by reducing the other relaxation parameter), a feature not possible in BGK schemes.
This is particularly helpful with schemes featuring a zero velocity (\lbmScheme{1}{3}, \lbmScheme{2}{9}, \emph{etc.}), which are widely-used and, in their BGK version, see their relaxation parameter stuck very close to one to achieve monotonicity, \confer{} \cite{aregba2024convergence}.
In practice, TRT schemes involve two relaxation parameters $\relaxationParameterSymmetric, \relaxationParameterAntiSymmetric\in(0, 2]$. Only $\relaxationParameterAntiSymmetric$ influences the numerical viscosity (as already pointed out in \cite{ginzburg2006variably}), which, as derived from the modified equation, is $\propto \spaceStep \bigl ( \frac{1}{\relaxationParameterAntiSymmetric} - \tfrac{1}{2}\bigr ) \times (\text{scheme-dependent non-linear term})$.\footnote{To be more precise, this is the numerical viscosity of the physical mode: one must keep in mind that several modes coexist in lattice Boltzmann schemes. However, when $\relaxationParameterSymmetric, \relaxationParameterAntiSymmetric\in(0, 2)$, the numerical modes decay to equilibrium, thus become ``slave'' of the physical mode.}
By tuning $\relaxationParameterSymmetric$ appropriately, TRT schemes can maintain monotonicity even as $\relaxationParameterAntiSymmetric$ approaches 2, providing a notable improvement over BGK schemes.

\subsection{Plan of the work}

The paper is structured as follows.
In \cref{sec:schemesTRT}, we describe the TRT schemes in detail.
\Cref{sec:monotonicity} defines monotonicity, establishes conditions for it to hold, and studies related properties.
Then, \cref{sec:convergence} draws consequences of monotonicity to prove convergence to the entropy solution of \eqref{eq:conservationLaw}.
Numerical experiments are presented in \cref{sec:numericalExperiments} to validate the theoretical results. Finally, conclusions and future directions are discussed in \cref{sec:Conclusions}.

\section{Two-relaxation-times lattice Boltzmann schemes}\label{sec:schemesTRT}

\subsection{Parameters}
\begin{itemize}
    \item We consider a space-step $\spaceStep>0$, with time-step $\timeStep$ linked \emph{via} $\timeStep = \spaceStep / \latticeVelocity$, where $\latticeVelocity>0$ is kept fixed whenever the limit $\spaceStep\to 0$ is considered.
    The space-grid is made up of $\spaceGridPoint{\vectorial{\indexSpace}} \definitionEquality\vectorial{\indexSpace}\spaceStep$ for $\vectorial{\indexSpace}\in\relatives^{\spatialDimensionality}$, whereas the time-grid is composed of $\timeGridPoint{\indexTime} \definitionEquality\indexTime\timeStep$ for $\indexTime\in\naturals$.
    It is convenient to think at $\spaceGridPoint{\vectorial{\indexSpace}}$ as the center of the cell $\cell{\vectorial{\indexSpace}} \definitionEquality\prod_{\indexDirection=1}^{\indexDirection=\spatialDimensionality} \bigl (\bigl (\indexSpace_{\indexDirection}-\tfrac{1}{2}\bigr )\spaceStep,\bigl (\indexSpace_{\indexDirection}+\tfrac{1}{2}\bigr )\spaceStep\bigr )$.
    \item We consider $\numberVelocities = 1 + 2\numberLinks$, where $\numberLinks \in \naturalsWithoutZero$ is the number of links, \idEst{} the number of mutually opposed discrete velocities, listed consecutively for the sake of readability:
    \begin{equation}\label{eq:choiceDiscreteVelocities}
        \vectorial{\discreteVelocityLetter}_1 = \vectorial{0}, \qquad 
        \vectorial{\discreteVelocityLetter}_{2\indexLink} = -\vectorial{\discreteVelocityLetter}_{2\indexLink + 1} \in \latticeVelocity\relatives^{\spatialDimensionality}, \qquad \indexLink \in \integerInterval{1}{\numberLinks}.
    \end{equation}
    \item To each discrete velocity $\vectorial{\discreteVelocityLetter}_{\indexVelocity}$ with $\indexVelocity\in\integerInterval{1}{\numberVelocities}$, we tie a distribution function $\distributionFunction_{\indexVelocity}$.
    For a link $\indexLink \in \integerInterval{1}{\numberLinks}$, one can consider symmetric and anti-symmetric decompositions ($ \symmetricDistributionFunction_{1} = \distributionFunction_1$ and $\antiSymmetricDistributionFunction_1 = 0$ are understood):
    \begin{equation*}
        \symmetricDistributionFunction_{2\indexLink} = \symmetricDistributionFunction_{2\indexLink + 1} \definitionEquality \tfrac{1}{2}(\distributionFunction_{2\indexLink} + \distributionFunction_{2\indexLink + 1}), \qquad \antiSymmetricDistributionFunction_{2\indexLink} = - \antiSymmetricDistributionFunction_{2\indexLink + 1} \definitionEquality \tfrac{1}{2}(\distributionFunction_{2\indexLink} - \distributionFunction_{2\indexLink + 1}).
    \end{equation*}
    \item For each discrete velocity indexed by $\indexVelocity\in\integerInterval{1}{\numberVelocities}$, we consider its equilibrium $\distributionFunctionLetter_{\indexVelocity}^{\atEquilibrium}: \reals \to \reals$.
    Its symmetric and anti-symmetric parts are defined analogously to discrete distribution functions.
    \item Finally, two relaxation parameters are present, one for the symmetric part of the distribution functions, called $\relaxationParameterSymmetric \in (0, 2]$, and one for the anti-symmetric part, denoted by $\relaxationParameterAntiSymmetric \in (0, 2]$.
\end{itemize}
In what follows, we try to be as consistent as possible regarding the use of indices. Indeed, $\indexDirection\in\integerInterval{1}{\spatialDimensionality}$ designates a Cartesian direction, $\vectorial{\indexSpace}\in\relatives^{\spatialDimensionality}$ the discrete space, $\indexTime\in\naturals$ the discrete time, $\indexVelocity\in\integerInterval{1}{\numberVelocities}$ the discrete velocity, and $\indexLink\in\integerInterval{1}{\numberLinks}$ a link.

\subsection{Algorithm}

\subsubsection{Initialization}

We initialize data at equilibrium, hence use 
\begin{equation}\label{eq:initializationEquilibrium}
    \distributionFunction_{\indexVelocity,\vectorial{\indexSpace}}^0 = \distributionFunctionLetter_{\indexVelocity}^{\atEquilibrium}\Bigl ( \dashint_{\cell{\vectorial{\indexSpace}}} \conservedVariable^{\initial}(\vectorial{\spaceVariable})\differential\vectorial{\spaceVariable} \Bigr), \qquad \indexVelocity\in\integerInterval{1}{\numberVelocities}.
\end{equation}
Let us note that the important question of how to initialize lattice Boltzmann schemes is addressed, in the linear framework, in \cite{bellotti2024initialisation} and the references therein. Without delving into alternative initialization strategies beyond \eqref{eq:initializationEquilibrium}, we highlight two key facts. 
First, use \eqref{eq:initializationEquilibrium} is the natural and straightforward choice, as it requires no additional knowledge about the structure of the numerical scheme, and is standard in the framework of kinetic schemes, \confer{} \cite{aregba2000discrete}, which share similarities with the lattice Boltzmann schemes. 
Second, this approach is sufficient with first-order schemes, see \cite{bellotti2024initialisation}, which is typically the case for monotone schemes (\cite[Theorem 3.1]{godlewski1991hyperbolic}) that we look at.
Considering higher-order corrections in $\bigO{\spaceStep}, \bigO{\spaceStep^2}, \dots$ in \eqref{eq:initializationEquilibrium} is useful solely with high-order schemes, \confer{} \cite{chen2024unified}.

\subsubsection{Collide-and-stream}

Once the initialization \eqref{eq:initializationEquilibrium} is provided, the algorithm proceeds in the following way, for every $\indexTime\in\naturals$ and $\vectorial{\indexSpace}\in\relatives^{\spatialDimensionality}$.

\paragraph{Relaxation}
Local to each point of the mesh.
Set $\conservedVariableDiscrete_{\vectorial{\indexSpace}}^{\indexTime} = \sum_{\indexVelocity = 1}^{\indexVelocity=\numberVelocities}\distributionFunction_{\indexVelocity, \vectorial{\indexSpace}}^{\indexTime}$ and perform $\distributionFunction_{\indexVelocity, \vectorial{\indexSpace}}^{\indexTime, \collided} = \distributionFunction_{\indexVelocity, \vectorial{\indexSpace}}^{\indexTime} + \relaxationParameterSymmetric (\symmetricDistributionFunctionLetter_{\indexVelocity}^{\atEquilibrium}(\conservedVariableDiscrete_{\vectorial{\indexSpace}}^{\indexTime}) - \symmetricDistributionFunction_{\indexVelocity, \vectorial{\indexSpace}}^{\indexTime}  ) +  \relaxationParameterAntiSymmetric (\antiSymmetricDistributionFunctionLetter_{\indexVelocity}^{\atEquilibrium}(\conservedVariableDiscrete_{\vectorial{\indexSpace}}^{\indexTime}) - \antiSymmetricDistributionFunction_{\indexVelocity, \vectorial{\indexSpace}}^{\indexTime}  )$ for $\indexVelocity\in\integerInterval{1}{\numberVelocities}$.
Recall that $\relaxationParameterSymmetric, \relaxationParameterAntiSymmetric \in (0, 2]$.
Otherwise written, the relaxation reads
\begin{alignat*}{1}
    \distributionFunction_{1, \vectorial{\indexSpace}}^{\indexTime, \collided} &=  ( 1 - \relaxationParameterSymmetric ) \distributionFunction_{1, \vectorial{\indexSpace}}^{\indexTime} + \relaxationParameterSymmetric  \distributionFunctionLetter_{1}^{\atEquilibrium}(\conservedVariableDiscrete_{\vectorial{\indexSpace}}^{\indexTime}), \\
    \distributionFunction_{2\indexLink, \vectorial{\indexSpace}}^{\indexTime, \collided} &= \bigl ( 1 - \tfrac{1}{2}(\relaxationParameterSymmetric + \relaxationParameterAntiSymmetric) \bigr ) \distributionFunction_{2\indexLink, \vectorial{\indexSpace}}^{\indexTime} + \tfrac{1}{2}(\relaxationParameterSymmetric + \relaxationParameterAntiSymmetric)  \distributionFunctionLetter_{2\indexLink}^{\atEquilibrium}(\conservedVariableDiscrete_{\vectorial{\indexSpace}}^{\indexTime}) - \tfrac{1}{2} (\relaxationParameterSymmetric - \relaxationParameterAntiSymmetric) (\distributionFunction_{2\indexLink + 1, \vectorial{\indexSpace}}^{\indexTime} - \distributionFunctionLetter_{2\indexLink + 1}^{\atEquilibrium}(\conservedVariableDiscrete_{\vectorial{\indexSpace}}^{\indexTime})), \\
    \distributionFunction_{2\indexLink + 1, \vectorial{\indexSpace}}^{\indexTime, \collided} &= \bigl ( 1 - \tfrac{1}{2}(\relaxationParameterSymmetric + \relaxationParameterAntiSymmetric) \bigr ) \distributionFunction_{2\indexLink + 1, \vectorial{\indexSpace}}^{\indexTime} + \tfrac{1}{2}(\relaxationParameterSymmetric + \relaxationParameterAntiSymmetric)  \distributionFunctionLetter_{2\indexLink + 1}^{\atEquilibrium}(\conservedVariableDiscrete_{\vectorial{\indexSpace}}^{\indexTime}) - \tfrac{1}{2} (\relaxationParameterSymmetric - \relaxationParameterAntiSymmetric) (\distributionFunction_{2\indexLink, \vectorial{\indexSpace}}^{\indexTime} - \distributionFunctionLetter_{2\indexLink}^{\atEquilibrium}(\conservedVariableDiscrete_{\vectorial{\indexSpace}}^{\indexTime})), 
\end{alignat*}
for $\indexLink\in\integerInterval{1}{\numberLinks}$.
For future use, the relaxation operator for the $\indexVelocity$-th distribution function is denoted by $\collisionOperator_{\indexVelocity}: \reals^{\numberVelocities}\to\reals$, with $\indexVelocity\in\integerInterval{1}{\numberVelocities}$. It is a non-linear function of the $\numberVelocities$ distribution functions such that, for $\indexVelocity\in\integerInterval{1}{\numberVelocities}$
\begin{equation}\label{eq:collision}
    \distributionFunction_{\indexVelocity, \vectorial{\indexSpace}}^{\indexTime,\collided} = \collisionOperator_{\indexVelocity}(\distributionFunction_{1, \vectorial{\indexSpace}}^{\indexTime}, \dots, \distributionFunction_{\numberVelocities, \vectorial{\indexSpace}}^{\indexTime}).
\end{equation}
\begin{remark}[BGK]
    Selecting $\relaxationParameterSymmetric = \relaxationParameterAntiSymmetric$, we obtain that $\distributionFunction_{\indexVelocity, \vectorial{\indexSpace}}^{\indexTime, \collided} =  ( 1 - \relaxationParameterSymmetric ) \distributionFunction_{\indexVelocity, \vectorial{\indexSpace}}^{\indexTime} + \relaxationParameterSymmetric  \distributionFunctionLetter_{\indexVelocity}^{\atEquilibrium}(\conservedVariableDiscrete_{\vectorial{\indexSpace}}^{\indexTime})$ for $\indexVelocity\in\integerInterval{1}{\numberVelocities}$, thus a BGK scheme.
    This setting has been analyzed in \cite{aregba2024convergence}, and forces both relaxation parameters to evolve in the same direction.
\end{remark}
\begin{remark}[Magic combination]
    When the magic combination $\relaxationParameterSymmetric + \relaxationParameterAntiSymmetric = 2$, extensively studied in \cite{ginzburg2008two, dellar2023magic, bellotti2024initialisation}, and \cite[Chapter 9, Section 4]{bellotti2023numerical}, holds---the relaxation becomes
    \begin{align*}
        \distributionFunction_{1, \vectorial{\indexSpace}}^{\indexTime, \collided} =  ( 1 - \relaxationParameterSymmetric ) \distributionFunction_{1, \vectorial{\indexSpace}}^{\indexTime} + \relaxationParameterSymmetric  \distributionFunctionLetter_{1}^{\atEquilibrium}(\conservedVariableDiscrete_{\vectorial{\indexSpace}}^{\indexTime}), \qquad
        \distributionFunction_{2\indexLink, \vectorial{\indexSpace}}^{\indexTime, \collided} &= \distributionFunctionLetter_{2\indexLink}^{\atEquilibrium}(\conservedVariableDiscrete_{\vectorial{\indexSpace}}^{\indexTime}) + (1-\relaxationParameterSymmetric) (\distributionFunction_{2\indexLink + 1, \vectorial{\indexSpace}}^{\indexTime} - \distributionFunctionLetter_{2\indexLink + 1}^{\atEquilibrium}(\conservedVariableDiscrete_{\vectorial{\indexSpace}}^{\indexTime})), \\
        \distributionFunction_{2\indexLink + 1, \vectorial{\indexSpace}}^{\indexTime, \collided} &=   \distributionFunctionLetter_{2\indexLink + 1}^{\atEquilibrium}(\conservedVariableDiscrete_{\vectorial{\indexSpace}}^{\indexTime}) + (1-\relaxationParameterSymmetric) (\distributionFunction_{2\indexLink, \vectorial{\indexSpace}}^{\indexTime} - \distributionFunctionLetter_{2\indexLink}^{\atEquilibrium}(\conservedVariableDiscrete_{\vectorial{\indexSpace}}^{\indexTime})), 
    \end{align*}
for $\indexLink\in\integerInterval{1}{\numberLinks}$.
This peculiar structure entails, \emph{inter alia}, a simple spectral structure of the scheme \cite{bellotti2024initialisation}.
This bond between relaxation parameters makes one of them decrease when the other increases, differently from the BGK approach.
\end{remark}

\paragraph{Transport}

Non-local but linear, made up of shifts on the grid according to the discrete velocities $\vectorial{\discreteVelocityLetter}_1, \dots, \vectorial{\discreteVelocityLetter}_{\numberVelocities}$, and which does not mingle the distribution functions:
\begin{equation}\label{eq:transport}
    \distributionFunction_{\indexVelocity, \vectorial{\indexSpace}}^{\indexTime + 1} = \distributionFunction_{\indexVelocity, \vectorial{\indexSpace} - \vectorial{\discreteVelocityLetter}_{\indexVelocity}/\latticeVelocity}^{\indexTime, \collided}.
\end{equation}

\subsection{Parameters ensuring consistency}

The algorithm is presented without a clear explanation of its connection to \eqref{eq:conservationLaw}, which raises the question of how to appropriately select the relaxation parameters $\relaxationParameterSymmetric$ and $\relaxationParameterAntiSymmetric$, the discrete velocities, and the equilibria. The following result demonstrates that a suitable choice of discrete velocities and equilibria is sufficient to ensure consistency. At this stage, consistency is considered for smooth solutions; the extension to weak solutions is addressed later.

\begin{proposition}[Consistency and modified equation]
    Let all parameters of the scheme be fixed as $\spaceStep$ goes to zero.
    Then, under the constraints
    \begin{equation}\label{eq:constraintsConsistency}
        \sum_{\indexVelocity = 1}^{\numberVelocities}\distributionFunctionLetter^{\atEquilibrium}_{\indexVelocity}(\conservedVariableDiscrete) = \conservedVariableDiscrete, \qquad \sum_{\indexLink = 1}^{\numberLinks} \discreteVelocityLetter_{2\indexLink, \indexDirection}(\distributionFunctionLetter_{2\indexLink}^{\atEquilibrium}(\conservedVariableDiscrete) - \distributionFunctionLetter_{2\indexLink + 1}^{\atEquilibrium}(\conservedVariableDiscrete)) = 2\sum_{\indexLink = 1}^{\numberLinks} \discreteVelocityLetter_{2\indexLink, \indexDirection}\antiSymmetricDistributionFunctionLetter_{2\indexLink}^{\atEquilibrium}(\conservedVariableDiscrete) = \flux_{\indexDirection}(\conservedVariableDiscrete),
    \end{equation}
    for $\indexDirection \in \integerInterval{1}{\spatialDimensionality}$, the numerical scheme is consistent, for smooth solutions, with \eqref{eq:conservationLaw}, according to the definitions by \cite{dubois2022nonlinear} and \cite{bellotti2023truncation}.
    Moreover, see \cite[Proposition 5]{dubois2022nonlinear} and \cite[Theorem 3.7]{bellotti2023truncation}, the modified equation up to second-order reads
    \begin{multline}
        \partial_{\timeVariable}\conservedVariable(\timeVariable, \vectorial{\spaceVariable}) + \sum\limits_{\indexDirection = 1}^{\spatialDimensionality} \partial_{\spaceVariable_{\indexDirection}}\flux_{\indexDirection}(\conservedVariable(\timeVariable, \vectorial{\spaceVariable})) = \frac{2\spaceStep}{\latticeVelocity} \Bigl ( \frac{1}{\relaxationParameterAntiSymmetric} - \frac{1}{2}\Bigr ) \\
        \times \sum_{\indexLink=1}^{\numberLinks} \sum_{\indexDirection=1}^{\spatialDimensionality} \discreteVelocityLetter_{2\indexLink, \indexDirection}\partial_{\spaceVariable_{\indexDirection}} \sum_{\tilde{\indexDirection}=1}^{\spatialDimensionality}  \Bigl ( \discreteVelocityLetter_{2\indexLink, \tilde{\indexDirection}}\partial_{\spaceVariable_{\tilde{\indexDirection}}}\symmetricDistributionFunctionLetter_{2\indexLink}^{\atEquilibrium}(\conservedVariable(\timeVariable, \vectorial{\spaceVariable})) - \frac{\differential\antiSymmetricDistributionFunctionLetter_{2\indexLink}^{\atEquilibrium}(\conservedVariable(\timeVariable, \vectorial{\spaceVariable}))}{\differential\conservedVariable}  \partial_{\spaceVariable_{\tilde{\indexDirection}}}\flux_{\tilde{\indexDirection}}(\conservedVariable(\timeVariable, \vectorial{\spaceVariable})) \Bigr ) + \bigO{\spaceStep^2}.\label{eq:modifiedEquation}
    \end{multline}
\end{proposition}

Although \strong{modified equations} \cite{warming1974modified} are derived for smooth solutions, it is known \cite{teng2010error} that, at least in the linear one-dimensional case, they provide valuable insights into monotone schemes even when the \strong{initial data are not smooth}.

\begin{remark}[Magic combination]
    In the case where $\relaxationParameterSymmetric + \relaxationParameterAntiSymmetric = 2$, the situation becomes even clearer and does not require the use of Taylor expansions. Specifically, \cite{ginzburg:tel-02591565}, later reformulated by \cite{bellotti2022finite}, demonstrates that $\conservedVariableDiscrete_{\vectorial{\indexSpace}}^{\indexTime}$, obtained from the lattice Boltzmann scheme, satisfies the following two-steps Finite Difference scheme:
    \begin{multline*}
        \frac{1}{\relaxationParameterAntiSymmetric\timeStep} (\conservedVariableDiscrete_{\vectorial{\indexSpace}}^{\indexTime + 1} + (\relaxationParameterAntiSymmetric-2)\conservedVariableDiscrete_{\vectorial{\indexSpace}}^{\indexTime} + (1-\relaxationParameterAntiSymmetric)\conservedVariableDiscrete_{\vectorial{\indexSpace}}^{\indexTime-1}) = \frac{2\latticeVelocity}{\spaceStep} \sum_{\indexLink = 1}^{\numberLinks} (\antiSymmetricDistributionFunctionLetter_{2\indexLink}^{\atEquilibrium}(\conservedVariableDiscrete_{\vectorial{\indexSpace}-\vectorial{\discreteVelocityLetter}_{2\indexLink}/\latticeVelocity}^{\indexTime}) - \antiSymmetricDistributionFunctionLetter_{2\indexLink}^{\atEquilibrium}(\conservedVariableDiscrete_{\vectorial{\indexSpace}+\vectorial{\discreteVelocityLetter}_{2\indexLink}/\latticeVelocity}^{\indexTime}))\\
        + \Bigl( \frac{1}{\relaxationParameterAntiSymmetric} - \frac{1}{2}\Bigr )\frac{\latticeVelocity}{\spaceStep} \sum_{\indexLink = 1}^{\numberLinks} ( \symmetricDistributionFunctionLetter_{2\indexLink}^{\atEquilibrium}(\conservedVariableDiscrete_{\vectorial{\indexSpace}-\vectorial{\discreteVelocityLetter}_{2\indexLink}/\latticeVelocity}^{\indexTime}) - 2\symmetricDistributionFunctionLetter_{2\indexLink}^{\atEquilibrium}(\conservedVariableDiscrete_{\vectorial{\indexSpace}}^{\indexTime})+ \symmetricDistributionFunctionLetter_{2\indexLink}^{\atEquilibrium}(\conservedVariableDiscrete_{\vectorial{\indexSpace}+\vectorial{\discreteVelocityLetter}_{2\indexLink}/\latticeVelocity}^{\indexTime})).
    \end{multline*}
    The left-hand side is a time integrator---second-order accurate when $\relaxationParameterAntiSymmetric = 2$.
    The first term on the right-hand side is consistent, under \eqref{eq:constraintsConsistency}, with minus the flux of \eqref{eq:conservationLaw}.
    Finally, the last term on the right-hand side is a link-wise dissipation term, which can be made small when $\relaxationParameterAntiSymmetric\to 2$.
\end{remark}

\subsection{Link with relaxation systems}

Let us finish the section by stressing that the numerical scheme can be seen as a discretization of the following relaxation system
\begin{equation*}
    \begin{cases}
        \partial_{\timeVariable} \distributionFunctionLetter_1 = \frac{1}{\relaxationTimeSymmetric}(\distributionFunctionLetter_1^{\atEquilibrium}(\conservedVariable)-\distributionFunctionLetter_1), \\
        \partial_{\timeVariable} \distributionFunctionLetter_{2\indexLink} + \vectorial{\discreteVelocityLetter}_{2\indexLink}\cdot \nabla_{\vectorial{\spaceVariable}}\distributionFunctionLetter_{2\indexLink} = \tfrac{1}{2}\bigl ( \frac{1}{\relaxationTimeSymmetric} +  \frac{1}{\relaxationTimeAntiSymmetric}\bigr ) (\distributionFunctionLetter_{2\indexLink}^{\atEquilibrium}(\conservedVariable) - \distributionFunctionLetter_{2\indexLink}) + \tfrac{1}{2}\bigl ( \frac{1}{\relaxationTimeSymmetric} -  \frac{1}{\relaxationTimeAntiSymmetric}\bigr ) ( \distributionFunctionLetter_{2\indexLink+1}^{\atEquilibrium}(\conservedVariable) - \distributionFunctionLetter_{2\indexLink+1}), \\
        \partial_{\timeVariable} \distributionFunctionLetter_{2\indexLink + 1} - \vectorial{\discreteVelocityLetter}_{2\indexLink}\cdot \nabla_{\vectorial{\spaceVariable}}\distributionFunctionLetter_{2\indexLink + 1} = \tfrac{1}{2}\bigl ( \frac{1}{\relaxationTimeSymmetric} +  \frac{1}{\relaxationTimeAntiSymmetric}\bigr ) (\distributionFunctionLetter_{2\indexLink+1}^{\atEquilibrium}(\conservedVariable) - \distributionFunctionLetter_{2\indexLink+1}) + \tfrac{1}{2}\bigl ( \frac{1}{\relaxationTimeSymmetric} -  \frac{1}{\relaxationTimeAntiSymmetric}\bigr ) (\distributionFunctionLetter_{2\indexLink}^{\atEquilibrium}(\conservedVariable)-\distributionFunctionLetter_{2\indexLink}),
    \end{cases}
\end{equation*}
with $\indexLink\in\integerInterval{1}{\numberLinks}$ and $\conservedVariable=\sum_{\indexVelocity=1}^{\indexVelocity=\numberVelocities}\distributionFunctionLetter_{\indexVelocity}$, if the left-hand side is discretized with any one-step consistent scheme, and the right-hand side using an explicit Euler method.
This holds upon identifying $\relaxationParameterSymmetric = \timeStep/\relaxationTimeSymmetric$ and $\relaxationParameterAntiSymmetric = \timeStep/\relaxationTimeAntiSymmetric$.
Using symmetric and anti-symmetric parts as in the discrete setting, we obtain, using \eqref{eq:constraintsConsistency}, the equivalent form 
\begin{equation*}
    \begin{cases}
        \partial_{\timeVariable} \conservedVariable + 2 \sum_{\indexLink=1}^{\indexLink = \numberLinks} \vectorial{\discreteVelocityLetter}_{2\indexLink}\cdot \nabla_{\vectorial{\spaceVariable}} a_{2\indexLink} = 0, \\
        \partial_{\timeVariable} s_{2\indexLink} + \vectorial{\discreteVelocityLetter}_{2\indexLink}\cdot \nabla_{\vectorial{\spaceVariable}} a_{2\indexLink} = \frac{1}{\relaxationTimeSymmetric}(s_{2\indexLink}^{\atEquilibrium}(\conservedVariable)-s_{2\indexLink}), \\
        \partial_{\timeVariable} a_{2\indexLink} + \vectorial{\discreteVelocityLetter}_{2\indexLink}\cdot \nabla_{\vectorial{\spaceVariable}} s_{2\indexLink} = \frac{1}{\relaxationTimeAntiSymmetric}(a_{2\indexLink}^{\atEquilibrium}(\conservedVariable)-a_{2\indexLink}),\qquad \indexLink\in\integerInterval{1}{\numberLinks}.
    \end{cases}
\end{equation*}

\section{Monotonicity}\label{sec:monotonicity}

\subsection{Further assumptions}

Assume that the equilibria split into a linear and a non-linear part proportional to the fluxes of \eqref{eq:conservationLaw}:
\begin{equation}\label{eq:equilibriaForm}
    \distributionFunctionLetter_{\indexVelocity}^{\atEquilibrium}(\conservedVariableDiscrete) = \equilibriumCoefficientLinear_{\indexVelocity} \conservedVariableDiscrete + \sum_{\indexDirection = 1}^{\spatialDimensionality}\equilibriumCoefficientFlux_{\indexVelocity, \indexDirection}\flux_{\indexDirection}(\conservedVariableDiscrete), \qquad \indexVelocity\in\integerInterval{1}{\numberVelocities}.
\end{equation}
This assumption is also made in \cite{aregba2024convergence}, and can be traced back (at least) to \cite[Equation (2.22)]{natalini1998discrete} and \cite[Equation (3.34)]{bouchut1999construction} in the context of (BGK) relaxation models.
Introducing more general kinds of non-linearities would, on the one hand, prevent from obtaining the forthcoming results for a broad class of numerical schemes, and, on the other hand, poses the interesting yet difficult problem of finding a criterion to devise these terms. For the interested readers, we highlight that kinetic schemes which do not rely on equilibria of the form \eqref{eq:equilibriaForm}, for instance ``flux-decomposition schemes'', have been investigated in \cite{aregba2000discrete}.
The constraints \eqref{eq:constraintsConsistency} thus become
\begin{align}
    \sum_{\indexVelocity = 1}^{\numberVelocities} \equilibriumCoefficientLinear_{\indexVelocity} &= 1, \qquad \sum_{\indexVelocity = 1}^{\numberVelocities} \equilibriumCoefficientFlux_{\indexVelocity, \indexDirection} = 0, \quad \indexDirection\in\integerInterval{1}{\spatialDimensionality}, \nonumber\\
    \sum_{\indexLink = 1}^{\numberLinks}\discreteVelocityLetter_{2\indexLink, \indexDirection} (\equilibriumCoefficientLinear_{2\indexLink} - \equilibriumCoefficientLinear_{2\indexLink + 1}) &= 0, \qquad 
    \sum_{\indexLink = 1}^{\numberLinks}\discreteVelocityLetter_{2\indexLink, \indexDirection} (\equilibriumCoefficientFlux_{2\indexLink, p} - \equilibriumCoefficientFlux_{2\indexLink + 1, p}) = \delta_{\indexDirection, p},\label{eq:constraintsConsistency2}
\end{align}
for $\indexDirection, p \in\integerInterval{1}{\spatialDimensionality}$.
It is also natural to request some symmetry along links, namely that 
\begin{equation}
    \equilibriumCoefficientLinear_{2\indexLink} = \equilibriumCoefficientLinear_{2\indexLink + 1} \qquad\text{and}\qquad \equilibriumCoefficientFlux_{2\indexLink, \indexDirection} = - \equilibriumCoefficientFlux_{2\indexLink + 1, \indexDirection} \geq 0,
\end{equation}
which we always assume in what follows.
Then, \eqref{eq:constraintsConsistency2} becomes
\begin{equation}\label{eq:constraintsConsistency3}
    \equilibriumCoefficientLinear_{1} + 2 \sum_{\indexLink = 1}^{\numberLinks}  \equilibriumCoefficientLinear_{2\indexLink}  = 1, \qquad  \equilibriumCoefficientFlux_{1, \indexDirection} = 0, 
    \qquad 2 \sum_{\indexLink = 1}^{\numberLinks}\discreteVelocityLetter_{2\indexLink, \indexDirection} \equilibriumCoefficientFlux_{2\indexLink, p}  = \delta_{\indexDirection, p}, \qquad \indexDirection, p \in\integerInterval{1}{\spatialDimensionality},
\end{equation}
whence the linear part of the equilibrium is symmetric, and the non-linear part is anti-symmetric.
The assumptions introduced in this section are assumed to hold throughout the paper and \strong{shall not be recalled anymore}.

\begin{remark}[On the case $\equilibriumCoefficientLinear_1 = 0$: a sort of \lbmScheme{\spatialDimensionality}{2\numberLinks} scheme]\label{rem:D1Q2}
    Let us discuss the case $\equilibriumCoefficientLinear_1 = 0$ in detail.
    Equation \eqref{eq:initializationEquilibrium} entails $\distributionFunction_{1, \vectorial{\indexSpace}}^{\indexTime} = 0$ for all $\indexTime\in\naturals$ and $\vectorial{\indexSpace}\in\relatives^{\spatialDimensionality}$: we could indeed avoid storing this unknown and $\conservedVariableDiscrete_{\vectorial{\indexSpace}}^{\indexTime}=\sum_{\indexVelocity=2}^{\indexVelocity = \numberVelocities} \distributionFunction_{\indexVelocity, \vectorial{\indexSpace}}^{\indexTime}$.
    The scheme practically becomes a \lbmScheme{\spatialDimensionality}{2\numberLinks}, with only an even number of pairwise opposed velocities, and no zero velocity.
\end{remark}

\subsection{Monotonicity of the relaxation}

Since one of the main aims of monotonicity is to ensure that the discrete solution remains within certain compact sets, we define $\maximumInitialDatum \definitionEquality \lVert \conservedVariable^{\circ} \rVert_{\lebesgueSpace{\infty}}$, so that the conserved moment stays within the interval $[-\maximumInitialDatum, \maximumInitialDatum]$. We also introduce the compact set 
\begin{equation*}
    \invariantCompactSetDistributions \definitionEquality \prod_{\indexVelocity = 1}^{\numberVelocities} [\distributionFunctionLetter_{\indexVelocity}^{\atEquilibrium}(-\maximumInitialDatum), \distributionFunctionLetter_{\indexVelocity}^{\atEquilibrium}(\maximumInitialDatum)],
\end{equation*}
where the distribution functions are to remain.
Notice that---thanks to \eqref{eq:constraintsConsistency}---if $(\distributionFunction_1, \dots, \distributionFunction_{\numberVelocities})\in\invariantCompactSetDistributions$, then $\conservedVariableDiscrete = \sum_{\indexVelocity=1}^{\indexVelocity = \numberVelocities}\distributionFunction_{\indexVelocity} \in [-\maximumInitialDatum, \maximumInitialDatum]$.

\begin{definition}[Monotone relaxation]\label{def:monotoneRelaxation}
    We say that the relaxation operator 
    \begin{equation*}
        \vectorial{\collisionOperator}(\distributionFunction_1, \dots, \distributionFunction_{\numberVelocities}) \definitionEquality 
        \begin{pmatrix}
            \collisionOperator_1(\distributionFunction_1, \dots, \distributionFunction_{\numberVelocities})\\
            \vdots\\
            \collisionOperator_{\numberVelocities}(\distributionFunction_1, \dots, \distributionFunction_{\numberVelocities})
        \end{pmatrix} : 
        \reals^{\numberVelocities}\to\reals^{\numberVelocities}
    \end{equation*}
    is monotone non-decreasing over $\invariantCompactSetDistributions$ if, for all $(\distributionFunction_1, \dots, \distributionFunction_{\numberVelocities})\in\invariantCompactSetDistributions$, it is non-decreasing with respect to each of its arguments.
\end{definition}

This definition of monotonicity is fundamentally different from the one used for Finite Difference/Volume schemes: in lattice Boltzmann schemes, the arguments of the scheme are not the same unknown at different grid points, but rather different distribution functions at different grid points (\confer{}  the transport step \eqref{eq:transport}). 
Since, under the assumptions made so far, $\vectorial{\collisionOperator} \in \continuousSpace{1}(\reals^{\numberVelocities})$, verifying \cref{def:monotoneRelaxation} reduces to ensuring that its Jacobian matrix has only non-negative entries.
\begin{proposition}[Monotonicity conditions]\label{prop:monotonicityConditions}
    Under the constraints 
    \begin{align}
        &\relaxationParameterSymmetric\equilibriumCoefficientLinear_1 \geq \max(0, \relaxationParameterSymmetric-1), \label{eq:zeroVelocityCondition}\\
        &\relaxationParameterAntiSymmetric \max_{\conservedVariableDiscrete\in[-\maximumInitialDatum, \maximumInitialDatum]}\Bigl | \sum_{\indexDirection = 1}^{\spatialDimensionality}\equilibriumCoefficientFlux_{2\indexLink, \indexDirection}\flux_{\indexDirection}' (\conservedVariableDiscrete) \Bigr | \leq \relaxationParameterSymmetric\equilibriumCoefficientLinear_{2\indexLink} + \tfrac{1}{2}\min (\underbrace{2-\relaxationParameterSymmetric-\relaxationParameterAntiSymmetric}_{\text{magic-vanish.}}, 0, \underbrace{\relaxationParameterAntiSymmetric-\relaxationParameterSymmetric}_{\text{BGK-van.}}), \label{eq:blockCondition}
    \end{align}
    for $\indexLink\in\integerInterval{1}{\numberLinks}$, the relaxation operator is monotone non-decreasing.
\end{proposition}
Notice that \eqref{eq:zeroVelocityCondition} is the constraint pertaining to the zero-velocity distribution function, whereas \eqref{eq:blockCondition} concern each link of opposite velocities present in the scheme.

\begin{remark}[BGK \emph{vs.} TRT]
    The zero-velocity condition \eqref{eq:zeroVelocityCondition}, which also applies in the BGK case, can be very restrictive; \idEst{}, it requires $\relaxationParameterSymmetric$ to be very close to one, particularly when $\equilibriumCoefficientLinear_1 \approx 0$. This is one of the advantages of TRT models: by taking $\relaxationParameterSymmetric < 1$ to satisfy \eqref{eq:zeroVelocityCondition}, one can still increase $\relaxationParameterAntiSymmetric$ above one using \eqref{eq:blockCondition}.
\end{remark}

\begin{remark}[On the case $\equilibriumCoefficientLinear_1 = 0$]
    Whenever $\equilibriumCoefficientLinear_1 = 0$, \confer{} \Cref{rem:D1Q2}, we observe that \eqref{eq:zeroVelocityCondition} imposes $\relaxationParameterSymmetric \leq 1$. 
    This condition is highly (and perhaps unnecessarily) restrictive, arising from the requirement that the relaxation be monotone with respect to $\distributionFunction_1$, without accounting for the fact that $\distributionFunction_{1, \vectorial{\indexSpace}}^{\indexTime} \equiv 0$ due to \eqref{eq:initializationEquilibrium}. 
    As such, this constraint can be safely disregarded. 
    This observation highlights a key point: when $\equilibriumCoefficientLinear_1 \approx 0$ and we initialize at equilibrium, the limitation imposed by \eqref{eq:zeroVelocityCondition} can be overly restrictive (particularly in the BGK setting), since the zero-velocity plays a minimal role. 
    While monotone schemes possess a sufficient amount of numerical diffusion, we see that the numerical diffusion obtained by \eqref{eq:modifiedEquation} does not indeed depend on the choice of $\distributionFunctionLetter_1^{\atEquilibrium}$, thus on $\equilibriumCoefficientLinear_1$.
    This situation is analogous to what occurs with linear multi-step methods for ODEs \cite{hundsdorfer2003monotonicity, hundsdorfer2006monotonicity}: despite the presence of negative coefficients, appropriate initializations can still ensure desirable monotonicity properties.
\end{remark}
\begin{remark}[On the smoothness of the fluxes]
    Notice that we assumed the fluxes $\flux_{\indexDirection}$ to be $\continuousSpace{1}$.
    Still, we could just work with (locally) Lipschitz continuous fluxes and still pursue all the proofs to come, at the price of more involved notations.
\end{remark}

\begin{proof}[Proof of \Cref{prop:monotonicityConditions}]
    The relaxation operator reads, inserting \eqref{eq:equilibriaForm}, and taking all the assumptions into account:
    \begin{alignat*}{1}
        \collisionOperator_1(\distributionFunction_1, \dots, \distributionFunction_{\numberVelocities}) &= ( 1 - \relaxationParameterSymmetric + \relaxationParameterSymmetric\equilibriumCoefficientLinear_1) \distributionFunction_{1} + \relaxationParameterSymmetric \equilibriumCoefficientLinear_1 \sum_{\indexVelocity=2}^{\numberVelocities}\distributionFunction_{\indexVelocity}, \\
        \collisionOperator_{2\indexLink}(\distributionFunction_1, \dots, \distributionFunction_{\numberVelocities}) &=  \bigl ( 1 - \tfrac{1}{2}(\relaxationParameterSymmetric + \relaxationParameterAntiSymmetric) + \relaxationParameterSymmetric \equilibriumCoefficientLinear_{2\indexLink} \bigr )\distributionFunction_{2\indexLink} + \bigl ( - \tfrac{1}{2}(\relaxationParameterSymmetric - \relaxationParameterAntiSymmetric) + \relaxationParameterSymmetric \equilibriumCoefficientLinear_{2\indexLink}  \bigr ) \distributionFunction_{2\indexLink+1}
        \\
        &\qquad\qquad\qquad+\relaxationParameterSymmetric \equilibriumCoefficientLinear_{2\indexLink} \sum_{\indexVelocity \neq 2\indexLink, 2\indexLink+1}\distributionFunction_{\indexVelocity} 
        + \relaxationParameterAntiSymmetric \sum_{\indexDirection = 1}^{\spatialDimensionality}\equilibriumCoefficientFlux_{2\indexLink, \indexDirection}\flux_{\indexDirection} \bigl ( \sum_{\indexVelocity = 1}^{\numberVelocities}\distributionFunction_{\indexVelocity} \bigr ), \\
        \collisionOperator_{2\indexLink+1}(\distributionFunction_1, \dots, \distributionFunction_{\numberVelocities}) &=  \bigl ( 1 - \tfrac{1}{2}(\relaxationParameterSymmetric + \relaxationParameterAntiSymmetric) + \relaxationParameterSymmetric \equilibriumCoefficientLinear_{2\indexLink} \bigr )\distributionFunction_{2\indexLink + 1} + \bigl ( - \tfrac{1}{2}(\relaxationParameterSymmetric - \relaxationParameterAntiSymmetric) + \relaxationParameterSymmetric \equilibriumCoefficientLinear_{2\indexLink}  \bigr ) \distributionFunction_{2\indexLink}
        \\
        &\qquad\qquad\qquad+\relaxationParameterSymmetric \equilibriumCoefficientLinear_{2\indexLink} \sum_{\indexVelocity \neq 2\indexLink, 2\indexLink+1}\distributionFunction_{\indexVelocity} 
        - \relaxationParameterAntiSymmetric \sum_{\indexDirection = 1}^{\spatialDimensionality}\equilibriumCoefficientFlux_{2\indexLink, \indexDirection}\flux_{\indexDirection} \bigl ( \sum_{\indexVelocity = 1}^{\numberVelocities}\distributionFunction_{\indexVelocity} \bigr ),
    \end{alignat*}
    with $\indexLink\in\integerInterval{1}{\numberLinks}$.
    For the zero-velocity, simple computations convey 
    \begin{equation*}
        \partial_{\distributionFunction_{\indexVelocity}}\collisionOperator_1(\distributionFunction_1, \dots, \distributionFunction_{\numberVelocities}) = 
        \begin{cases}
            1 - \relaxationParameterSymmetric + \relaxationParameterSymmetric\equilibriumCoefficientLinear_1, \qquad &\text{if}\quad \indexVelocity = 1, \\
            \relaxationParameterSymmetric\equilibriumCoefficientLinear_1, \qquad &\text{otherwise}.
        \end{cases}
    \end{equation*}
    For any link $\indexLink\in\integerInterval{1}{\numberLinks}$, using $\conservedVariableDiscrete=\sum_{\indexVelocity=1}^{\indexVelocity=\numberVelocities} \distributionFunction_{\indexVelocity}$, we have 
    \begin{align*}
        \partial_{\distributionFunction_{\indexVelocity}}\collisionOperator_{2\indexLink} &=
        \begin{cases}
            1 - \tfrac{1}{2}(\relaxationParameterSymmetric + \relaxationParameterAntiSymmetric) + &\relaxationParameterSymmetric \equilibriumCoefficientLinear_{2\indexLink} + \relaxationParameterAntiSymmetric \sum_{\indexDirection = 1}^{\indexDirection=\spatialDimensionality}\equilibriumCoefficientFlux_{2\indexLink, \indexDirection}\flux'_{\indexDirection}(\conservedVariableDiscrete), \qquad \text{if}\quad \indexVelocity = 2\indexLink, \\
            - \tfrac{1}{2}(\relaxationParameterSymmetric - \relaxationParameterAntiSymmetric) + &\relaxationParameterSymmetric \equilibriumCoefficientLinear_{2\indexLink}  + \relaxationParameterAntiSymmetric \sum_{\indexDirection = 1}^{\indexDirection=\spatialDimensionality}\equilibriumCoefficientFlux_{2\indexLink, \indexDirection}\flux'_{\indexDirection}(\conservedVariableDiscrete), \qquad \text{if}\quad \indexVelocity = 2\indexLink + 1,\\
            &\relaxationParameterSymmetric \equilibriumCoefficientLinear_{2\indexLink} + \relaxationParameterAntiSymmetric \sum_{\indexDirection = 1}^{\indexDirection=\spatialDimensionality}\equilibriumCoefficientFlux_{2\indexLink, \indexDirection}\flux'_{\indexDirection}(\conservedVariableDiscrete), \qquad \text{otherwise}.
        \end{cases}\\
        \partial_{\distributionFunction_{\indexVelocity}}\collisionOperator_{2\indexLink + 1}&=
        \begin{cases}
            1 - \tfrac{1}{2}(\relaxationParameterSymmetric + \relaxationParameterAntiSymmetric) + &\relaxationParameterSymmetric \equilibriumCoefficientLinear_{2\indexLink} - \relaxationParameterAntiSymmetric \sum_{\indexDirection = 1}^{\indexDirection=\spatialDimensionality}\equilibriumCoefficientFlux_{2\indexLink, \indexDirection}\flux'_{\indexDirection}(\conservedVariableDiscrete), \qquad \text{if}\quad \indexVelocity = 2\indexLink + 1, \\
            - \tfrac{1}{2}(\relaxationParameterSymmetric - \relaxationParameterAntiSymmetric) + &\relaxationParameterSymmetric \equilibriumCoefficientLinear_{2\indexLink}  - \relaxationParameterAntiSymmetric \sum_{\indexDirection = 1}^{\indexDirection=\spatialDimensionality}\equilibriumCoefficientFlux_{2\indexLink, \indexDirection}\flux'_{\indexDirection}(\conservedVariableDiscrete), \qquad \text{if}\quad \indexVelocity = 2\indexLink,\\
            &\relaxationParameterSymmetric \equilibriumCoefficientLinear_{2\indexLink} - \relaxationParameterAntiSymmetric \sum_{\indexDirection = 1}^{\indexDirection=\spatialDimensionality}\equilibriumCoefficientFlux_{2\indexLink, \indexDirection}\flux'_{\indexDirection}(\conservedVariableDiscrete), \qquad \text{otherwise}.
        \end{cases}
    \end{align*}
    We group the conditions on the non-negativity of the Jacobian cleverly, following the proof of \cite[Proposition 3]{dubois2020notion}.
    Start by the first row of the Jacobian: zero velocity.
    \begin{equation}\label{eq:constraintZeroVel}
        \begin{cases}
            \partial_{\distributionFunction_{1}}\collisionOperator_1(\distributionFunction_1, \dots, \distributionFunction_{\numberVelocities}) &\geq 0, \\
            \partial_{\distributionFunction_{\indexVelocity}}\collisionOperator_1(\distributionFunction_1, \dots, \distributionFunction_{\numberVelocities})&\geq0, \quad \indexVelocity\in\integerInterval{2}{\numberVelocities}
        \end{cases}\qquad \Longleftrightarrow \qquad 
        \relaxationParameterSymmetric\equilibriumCoefficientLinear_1 \geq \max(0, \relaxationParameterSymmetric-1).
    \end{equation}
    Let $\indexLink\in\integerInterval{1}{\numberLinks}$ be any link.
    Consider the sensitivity of members of the $\indexLink$-th link with respect to themselves:
        \begin{equation}\label{eq:constraintWithinLinkOne}
            \begin{cases}
                \partial_{\distributionFunction_{2\indexLink}}\collisionOperator_{2\indexLink} &\geq 0, \\
                \partial_{\distributionFunction_{2\indexLink +1}}\collisionOperator_{2\indexLink + 1} &\geq 0,
            \end{cases}\qquad \Longleftrightarrow\qquad \relaxationParameterAntiSymmetric \Bigl | \sum_{\indexDirection = 1}^{\spatialDimensionality}\equilibriumCoefficientFlux_{2\indexLink, \indexDirection}\flux_{\indexDirection}' (\conservedVariableDiscrete) \Bigr | \leq   \relaxationParameterSymmetric\equilibriumCoefficientLinear_{2\indexLink} + \tfrac{1}{2} (2-\relaxationParameterSymmetric-\relaxationParameterAntiSymmetric).
        \end{equation}
    Go to members of the $\indexLink$-th link with respect to their sibling:
        \begin{equation}\label{eq:constraintWithinLinkTwo}
            \begin{cases}
                \partial_{\distributionFunction_{2\indexLink + 1}}\collisionOperator_{2\indexLink} &\geq 0, \\
                \partial_{\distributionFunction_{2\indexLink}}\collisionOperator_{2\indexLink + 1} &\geq 0,
            \end{cases}\qquad \Longleftrightarrow\qquad \relaxationParameterAntiSymmetric \Bigl | \sum_{\indexDirection = 1}^{\spatialDimensionality}\equilibriumCoefficientFlux_{2\indexLink, \indexDirection}\flux_{\indexDirection}' (\conservedVariableDiscrete) \Bigr | \leq   \relaxationParameterSymmetric\equilibriumCoefficientLinear_{2\indexLink} - \tfrac{1}{2}(\relaxationParameterSymmetric-\relaxationParameterAntiSymmetric).
        \end{equation}
    End with members of the $\indexLink$-th link with respect to any distribution function outside the link:
        \begin{equation}\label{eq:constraintOutsideLink}
            \begin{cases}
                \partial_{\distributionFunction_{\indexVelocity}}\collisionOperator_{2\indexLink}&\geq0, \quad \indexVelocity\in\integerInterval{1}{\numberVelocities}\smallsetminus \{2\indexLink, 2\indexLink + 1\}\\
                \partial_{\distributionFunction_{\indexVelocity}}\collisionOperator_{2\indexLink + 1}&\geq0, \quad \indexVelocity\in\integerInterval{1}{\numberVelocities}\smallsetminus \{2\indexLink, 2\indexLink + 1\}
            \end{cases}\quad \Longleftrightarrow \quad 
            \relaxationParameterAntiSymmetric \Bigl | \sum_{\indexDirection = 1}^{\spatialDimensionality}\equilibriumCoefficientFlux_{2\indexLink, \indexDirection}\flux_{\indexDirection}' (\conservedVariableDiscrete) \Bigr | \leq   \relaxationParameterSymmetric\equilibriumCoefficientLinear_{2\indexLink}.
        \end{equation}
    The left-hand side of the last three inequalities is the same: we gather them using the minimum of the right-hand sides.
    Eventually, since we consider $(\distributionFunction_1, \dots, \distributionFunction_{\numberVelocities})\in\invariantCompactSetDistributions$, we take the most restrictive condition on $[-\maximumInitialDatum, \maximumInitialDatum]$ to which $\conservedVariableDiscrete$ belongs.
\end{proof}

Considering all the parameters be given except for $\relaxationParameterSymmetric$ and $\relaxationParameterAntiSymmetric$, we can draw the area in the two-dimensional plane where the conditions by \Cref{prop:monotonicityConditions} hold.
\begin{definition}
    Let $\equilibriumCoefficientLinear_1, \dots, \equilibriumCoefficientLinear_{\numberVelocities}$ and $\equilibriumCoefficientFlux_{1, \indexDirection}, \dots, \equilibriumCoefficientFlux_{\numberVelocities, \indexDirection}$ for $\indexDirection\in\integerInterval{1}{\spatialDimensionality}$ and $\conservedVariable^{\initial}$ be given.
    We define 
    \begin{alignat*}{4}
        \monotoneZoneTwoD&\definitionEquality\{(\relaxationParameterSymmetric, \relaxationParameterAntiSymmetric)&&\in&&(0, 2]^2 \quad &&\text{such that} ~ \text{\eqref{eq:zeroVelocityCondition} and \eqref{eq:blockCondition} hold}\} \subset \reals^2, \\
        \monotoneSegmentBGK&\definitionEquality\{\relaxationParameterSymmetric=\relaxationParameterAntiSymmetric &&\in&&(0, 2] \quad &&\text{such that} ~ \text{\eqref{eq:zeroVelocityCondition} and \eqref{eq:blockCondition} hold}\} \subset \reals,
    \end{alignat*}
    which depend on $\conservedVariable^{\initial}$ (through $\maximumInitialDatum$), $\flux_1, \dots, \flux_{\spatialDimensionality}, \equilibriumCoefficientLinear_1, \dots, \equilibriumCoefficientLinear_{\numberVelocities}$, and $ \equilibriumCoefficientFlux_{1, \indexDirection}, \dots, \equilibriumCoefficientFlux_{\numberVelocities, \indexDirection}$ for $\indexDirection\in\integerInterval{1}{\spatialDimensionality}$.
\end{definition}
The curious reader may wish to jump straight to \cref{fig:D1Q3} to observe how these plots actually look like. 
We now begin proving ``rigidity'' results concerning the structure of $\monotoneZoneTwoD$. 
This area possesses a trivial yet interesting geometrical property:
\begin{lemma}[Convexity]\label{lemma:convexity}
    The set $\monotoneZoneTwoD$ is convex.
\end{lemma}
\begin{proof}
    Since the derivatives of the relaxation operator are affine in $\relaxationParameterSymmetric$ and $\relaxationParameterAntiSymmetric$, the inequality constraints that we obtain by \eqref{eq:constraintZeroVel}, \eqref{eq:constraintWithinLinkOne}, \eqref{eq:constraintWithinLinkTwo}, and \eqref{eq:constraintOutsideLink}---and which make up $\monotoneZoneTwoD$---define a finite number of half-spaces. Their intersection forms a convex polytope.
\end{proof}
This result means that if we find two points in $\monotoneZoneTwoD$, the segment connecting these two points also lays within $\monotoneZoneTwoD$. 
Moreover, $\argmin_{(\relaxationParameterSymmetric, \relaxationParameterAntiSymmetric)\in\monotoneZoneTwoD} \bigl ( \frac{1}{\relaxationParameterAntiSymmetric} - \tfrac{1}{2}\bigr ) \in \partial\monotoneZoneTwoD$.
The next question is whether $(1, 1) \in \monotoneZoneTwoD$, which is significant because a positive answer would imply that the equilibria are monotone functions. This assumption, discussed in \cite{aregba2024convergence}, is useful to employ Krushkov kinetic entropies. 
The following result shows, \emph{inter alia}, that if there exists at least one pair $(\relaxationParameterSymmetric, \relaxationParameterAntiSymmetric) \in \monotoneZoneTwoD$ (\idEst{}, $\monotoneZoneTwoD \neq \varnothing$), then the equilibria are necessarily monotone (see supplementary material for the proof).
\begin{proposition}[Monotonicity of the equilibria and BGK segment]\label{prop:monotonicityEquilibria}
    We have that $\monotoneZoneTwoD\neq\varnothing$ if and only if the equilibria are monotone non-decreasing, \idEst{}
    \begin{equation*}
        \text{for all }\conservedVariableDiscrete\in[-\maximumInitialDatum, \maximumInitialDatum], \qquad 
        \frac{\differential\distributionFunctionLetter_{\indexVelocity}^{\atEquilibrium}(\conservedVariableDiscrete)}{\differential\conservedVariableDiscrete}\geq 0, \qquad \indexVelocity\in\integerInterval{1}{\numberVelocities}.
    \end{equation*}
    Moreover, in this case, as in \cite[Proposition 2.1]{aregba2024convergence} 
    \begin{align}
        \monotoneSegmentBGK =  \Bigl (0, \min\Bigl ( \frac{1}{1-\equilibriumCoefficientLinear_1}, 
        &\frac{1}{1-\equilibriumCoefficientLinear_{2}+\max\limits_{\conservedVariableDiscrete\in[-\maximumInitialDatum, \maximumInitialDatum]} | \sum_{\indexDirection = 1}^{\indexDirection=\spatialDimensionality}\equilibriumCoefficientFlux_{2, \indexDirection}\flux_{\indexDirection}' (\conservedVariableDiscrete)  |}, \nonumber\\
        &\dots, \nonumber\\
        &\frac{1}{1-\equilibriumCoefficientLinear_{2\numberLinks}+\max\limits_{\conservedVariableDiscrete\in[-\maximumInitialDatum, \maximumInitialDatum]} | \sum_{\indexDirection = 1}^{\indexDirection=\spatialDimensionality}\equilibriumCoefficientFlux_{2\numberLinks, \indexDirection}\flux_{\indexDirection}' (\conservedVariableDiscrete)  |}\Bigr )\Bigr ] \supset (0, 1],\label{eq:inequalityOmegaBGK}
    \end{align}
    and $\equilibriumCoefficientLinear_{\indexVelocity}\in[0, 1]$ for $\indexVelocity\in\integerInterval{1}{\numberVelocities}$.
\end{proposition}

The next property (proof provided in the supplementary material) is crucial for the forthcoming proofs, particularly to ensure that the discrete solution converges, geometrically in time, to a $\bigO{\spaceStep}$-neighborhood of the equilibrium, and thus to the equilibrium as $\spaceStep \to 0$.
\begin{proposition}[Monotone schemes cannot have both/either $\relaxationParameterSymmetric=2$ and/or $\relaxationParameterAntiSymmetric = 2$]\label{prop:neverEqualToTwo}
    Let $(\relaxationParameterSymmetric, \relaxationParameterAntiSymmetric)\in\monotoneZoneTwoD$.
    Then, $\relaxationParameterSymmetric\neq 2$ and $\relaxationParameterAntiSymmetric\neq 2$.
\end{proposition}

We finally prove that if there is at least one point of $\monotoneZoneTwoD$ strictly above the BGK diagonal, then $\monotoneZoneTwoD$ is somehow ``symmetric'' around the BGK segment: the open BGK segment belongs to the interior of $\monotoneZoneTwoD$, so that there is also something strictly below the BGK diagonal. 
The proof is in the supplementary material.
\begin{proposition}\label{prop:symmetryBGK}
    Assume that $\monotoneZoneTwoD\cap \{(\relaxationParameterSymmetric, \relaxationParameterAntiSymmetric)~ \text{s.t.} ~ \relaxationParameterAntiSymmetric>\relaxationParameterSymmetric\}\neq \varnothing$.
    Then, we have that $\{(\relaxationParameter, \relaxationParameter)~ \text{s.t.} ~\relaxationParameter\in\mathring{\monotoneZoneTwoD}_{\textnormal{BGK}}\}\subset \mathring{\monotoneZoneTwoD}$, which entails that $\monotoneZoneTwoD\cap \{(\relaxationParameterSymmetric, \relaxationParameterAntiSymmetric)~ \text{s.t.} ~ \relaxationParameterAntiSymmetric<\relaxationParameterSymmetric\}\neq \varnothing$.
\end{proposition}

\begin{remark}
    Notice that we could face $\monotoneZoneTwoD\cap \{(\relaxationParameterSymmetric, \relaxationParameterAntiSymmetric)~ \text{s.t.} ~ \relaxationParameterAntiSymmetric<\relaxationParameterSymmetric\}\neq \varnothing$ but $\monotoneZoneTwoD\cap \{(\relaxationParameterSymmetric, \relaxationParameterAntiSymmetric)~ \text{s.t.} ~ \relaxationParameterAntiSymmetric>\relaxationParameterSymmetric\}= \varnothing$, thus $\{(\relaxationParameter, \relaxationParameter)~ \text{s.t.} ~\relaxationParameter\in\mathring{\monotoneZoneTwoD}_{\textnormal{BGK}}\}\subset \partial\monotoneZoneTwoD$.
    This happens when there exists $\indexLink\in\integerInterval{1}{\numberLinks}$ such that $ \max_{\conservedVariableDiscrete\in[-\maximumInitialDatum, \maximumInitialDatum]} | \sum_{\indexDirection = 1}^{\indexDirection=\spatialDimensionality}\equilibriumCoefficientFlux_{2\indexLink, \indexDirection}\flux_{\indexDirection}' (\conservedVariableDiscrete) | =\equilibriumCoefficientLinear_{2\indexLink}$.
\end{remark}

\begin{figure}[h]
    \begin{center}
        \begin{tikzpicture}
        \foreach \j in {0,...,1} {
            \draw[->] (4.5,5.5*\j+2) -- (6.5,5.5*\j+2) node[midway, above] {we deduce};
            \foreach \i in {0,...,1} {

                    \draw[->] (7*\i, 5.5*\j) -- (7*\i+4.25, 5.5*\j) node[right] {$\relaxationParameterSymmetric$}; 
                    \draw[->] (7*\i,5.5*\j) -- (7*\i, 5.5*\j+4.25) node[above] {$\relaxationParameterAntiSymmetric$};
                    \node at (7*\i, 5.5*\j-0.15) {$0$};
                    \node at (7*\i+4, 5.5*\j-0.15) {$2$};
                    \node at (7*\i-0.15, 5.5*\j+4) {$2$};
                    \node at (7*\i+2, 5.5*\j-0.15) {$1$};
                    \node at (7*\i-0.15, 5.5*\j+2) {$1$};

                    \draw[dotted] (7*\i,5.5*\j) -- (7*\i+4, 5.5*\j+4);
                    \draw[dash dot] (7*\i-0.05,5.5*\j+4) -- (7*\i+4, 5.5*\j+4);
                    \draw[dash dot] (7*\i+4, 5.5*\j-0.05) -- (7*\i+4, 5.5*\j+4);

                }
            }
            \node[above] at (2, 10) {Knowing};


            \fill[fill=colorTwo] (1,8.5) circle (2pt);

            \fill[color=colorOne, opacity=0.5] (7, 5.5) -- (9.5, 8) -- (1+7,8.5) -- cycle;
            \fill[color=colorOne, opacity=0.5] (7, 5.5) -- (9.5, 8) -- (9,6.5)  -- cycle;
            \fill[fill=colorTwo] (1+7,8.5) circle (2pt);
            \draw[ultra thick, color=colorThree] (7, 5.5) -- (9.5, 8);
            \fill[fill=colorThree] (9.5, 8) circle (2pt);
            \fill[fill=colorThree] (9, 7.5) circle (2pt);
            \fill[fill=cyan] (9,6.5) circle (2pt);

            \node[above, color=colorOne, opacity=0.5] at (9,8.5) {Lemma \refNoColor{lemma:convexity}};
            \node[right, color=colorThree] at (9.5, 8) {Prop. \refNoColor{prop:monotonicityEquilibria}};
            \node[below, color=cyan] at (9,6.2) {Prop. \refNoColor{prop:symmetryBGK}};

            \fill[fill=colorTwo] (2.5,1.5) circle (2pt);

            \fill[color=colorOne, opacity=0.5] (7, 0) -- (9.5, 2.5) -- (2.5+7,1.5)  -- cycle;
            \fill[fill=colorTwo] (2.5+7,1.5) circle (2pt);
            \draw[ultra thick, color=colorThree] (7, 0) -- (9.5, 2.5);
            \fill[fill=colorThree] (9.5, 2.5) circle (2pt);
            \fill[fill=colorThree] (9, 2) circle (2pt);

            \node[below, color=colorOne, opacity=0.5] at (2.5+7,0.75) {Lemma \refNoColor{lemma:convexity}};
            \node[right, color=colorThree] at (9.5, 2.5) {Prop. \refNoColor{prop:monotonicityEquilibria}};
            
        \end{tikzpicture}
    \end{center}\caption{\label{fig:deductions}Left: points known in $\monotoneZoneTwoD$. Right: deductions on points in $\monotoneZoneTwoD$ (colored) that we can obtain from the knowledge presented on the left, using the previous results.}
\end{figure}

\Cref{fig:deductions} visually resumes the \emph{a priori} knowledge that we can obtain on $\monotoneZoneTwoD$ from knowing one of its points above or below the BGK diagonal.
This adheres very well to the actual $\monotoneZoneTwoD$, see \Cref{fig:D1Q3} and \ref{fig:D2Q5}.

\section{Convergence}\label{sec:convergence}

We now prove the convergence of the discrete solution of the lattice Boltzmann scheme under the conditions by \Cref{prop:monotonicityConditions}.
We introduce the following notations.
\begin{align*}
    \distributionFunctionsAsFunctionSpace{\indexTime}(\vectorial{\spaceVariable}) &\definitionEquality \sum_{\vectorial{\indexSpace}\in\relatives^{\spatialDimensionality}} {(\distributionFunction_{1, \vectorial{\indexSpace}}^{\indexTime}, \dots, \distributionFunction_{\numberVelocities, \vectorial{\indexSpace}}^{\indexTime})}\indicatorFunction{\cell{\vectorial{\indexSpace}}} (\vectorial{\spaceVariable}), \qquad \conservedVariableDiscreteAsFunctionSpace{\indexTime} (\vectorial{\spaceVariable}) \definitionEquality (1, \dots, 1)\cdot \distributionFunctionsAsFunctionSpace{\indexTime}(\vectorial{\spaceVariable}), \\
    \distributionFunctionsAsFunction(\timeVariable, \vectorial{\spaceVariable}) &\definitionEquality \sum_{\indexTime\in\naturals}\distributionFunctionsAsFunctionSpace{\indexTime}(\vectorial{\spaceVariable})  \indicatorFunction{[\timeGridPoint{\indexTime}, \timeGridPoint{\indexTime+1})} (\timeVariable), \qquad \conservedVariableDiscreteAsFunction (\timeVariable, \vectorial{\spaceVariable}) \definitionEquality\sum_{\indexTime\in\naturals}\conservedVariableDiscreteAsFunctionSpace{\indexTime}(\vectorial{\spaceVariable})  \indicatorFunction{[\timeGridPoint{\indexTime}, \timeGridPoint{\indexTime+1})} (\timeVariable).
\end{align*}
The total variation $\totalVariation{\conservedVariable}$ of a function $\conservedVariable\in\lebesgueSpace{1}_{\text{loc}}(\reals^{\spatialDimensionality})$ follows the usual definition, so that the total variation of the discrete solution becomes
\begin{equation*}
    \totalVariation{\conservedVariableDiscreteAsFunctionSpace{\indexTime}} = \spaceStep^{\spatialDimensionality-1}\sum_{\vectorial{\indexSpace}\in\relatives^{\spatialDimensionality}}\sum_{\indexDirection=1}^{\spatialDimensionality} |\conservedVariableDiscrete_{\vectorial{\indexSpace}+\canonicalBasisVector{\indexDirection}}^{\indexTime} - \conservedVariableDiscrete_{\vectorial{\indexSpace}}^{\indexTime}|,
\end{equation*}
where $\canonicalBasisVector{\indexDirection}$ is the $\indexDirection$-th vector of the canonical basis of $\reals^{\spatialDimensionality}$.
In this way, the total variation for the distribution functions is 
\begin{equation*}
    \totalVariationVectorial{\distributionFunctionsAsFunctionSpace{\indexTime}} \definitionEquality \spaceStep^{\spatialDimensionality-1}\sum_{\vectorial{\indexSpace}\in\relatives^{\spatialDimensionality}}\sum_{\indexDirection=1}^{\spatialDimensionality} \sum_{\indexVelocity=1}^{\numberVelocities}|\distributionFunction_{\indexVelocity, \vectorial{\indexSpace}+\canonicalBasisVector{\indexDirection}}^{\indexTime} - \distributionFunction_{\indexVelocity, \vectorial{\indexSpace}}^{\indexTime}|.
\end{equation*}

\begin{remark}[Norm on the vector space $\reals^{\numberVelocities}$]\label{rem:normVectorSpace}
    Before proceeding, notice that we endow the vector space $\reals^{\numberVelocities}$ of the distribution functions at each time/space grid-point\footnote{We treat this as a column or row vector, depending on the context, for notational convenience.} with the $\petitLebesgueSpace{1}$-norm.
    While all norms on $\reals^{\numberVelocities}$ are equivalent, this choice is particularly convenient.
    For norms (\exempliGratia{}, the $\lebesgueSpace{1}$-norm), or a semi-norms (\exempliGratia{}, total variation) applied to vectors of $\reals^{\numberVelocities}$, we do not stress that the $\petitLebesgueSpace{1}$-norm is being used.
    Moreover, we may simply write $\lebesgueSpace{1}$ to refer to $\lebesgueSpace{1}(\reals^{\spatialDimensionality}, \reals^{\numberVelocities})$.
\end{remark}

\subsection{Convergence to a weak solution}

The steps of the proof are as in \cite{aregba2000discrete,aregba2024convergence}: maximum principle ($\lebesgueSpace{\infty}$-bounds), $\lebesgueSpace{1}$-contractivity, estimates on the total variation, and convergence to the equilibrium.

\subsubsection{Maximum principle}

The monotonicity of the relaxation entails that---upon initializing at equilibrium \eqref{eq:initializationEquilibrium}---the numerical schemes preserve a compact set.
\begin{proposition}[Invariant compact sets]\label{prop:compactInvariantSets}
    Let the conditions by \Cref{prop:monotonicityConditions} be satisfied.
    Then, for all $\indexTime \in \naturals$
    \begin{equation}\label{eq:compactInvariantSets}
        \forall \vectorial{\indexSpace}\in\relatives^{\spatialDimensionality}, \qquad (\distributionFunction_{1, \vectorial{\indexSpace}}^{\indexTime}, \dots, \distributionFunction_{\numberVelocities, \vectorial{\indexSpace}}^{\indexTime})\in \invariantCompactSetDistributions, \quad \text{and}\quad \conservedVariableDiscrete_{\vectorial{\indexSpace}}^{\indexTime} = \sum_{\indexVelocity = 1}^{\numberVelocities} \distributionFunction_{\indexVelocity, \vectorial{\indexSpace}}^{\indexTime} \in [-\maximumInitialDatum, \maximumInitialDatum].
    \end{equation}
    Therefore $\lVert \conservedVariableDiscreteAsFunction\rVert_{\lebesgueSpace{\infty}} \leq \maximumInitialDatum$ and $\lVert\distributionFunctionsAsFunction \rVert_{\lebesgueSpace{\infty}} \leq \sum_{\indexVelocity=1}^{\indexVelocity=\numberVelocities} \max(|\distributionFunction_{\indexVelocity}^{\atEquilibrium}(-\maximumInitialDatum)|, |\distributionFunction_{\indexVelocity}^{\atEquilibrium}(\maximumInitialDatum)|)$.
\end{proposition}
\begin{remark}[Sufficient \emph{vs.} necessary conditions]
    For the proofs to come, it is crucial---this cannot be emphasized enough---that \eqref{eq:compactInvariantSets} holds.
    However, in contrast to one-step scalar schemes, \confer{} \cite{godlewski1991hyperbolic}, the conditions by \Cref{prop:monotonicityConditions} are sufficient, but are often not necessary.
    Indeed, the relaxation phase could be decreasing with respect to one argument, yet \eqref{eq:compactInvariantSets} could still hold.
    While the scheme may lack monotonicity when viewed over two time-steps, \idEst{} mapping $(\distributionFunction_1^{\indexTime}, \dots, \distributionFunction_{\numberVelocities}^{\indexTime})\mapsto (\distributionFunction_1^{\indexTime +1}, \dots, \distributionFunction_{\numberVelocities}^{\indexTime+1})$, it might exhibit monotonicity-like properties when acting on the initial datum, sending $(\distributionFunction_1^{0}, \dots, \distributionFunction_{\numberVelocities}^{0})\mapsto (\distributionFunction_1^{\indexTime}, \dots, \distributionFunction_{\numberVelocities}^{\indexTime})$, especially when \eqref{eq:initializationEquilibrium} holds.
    Gaps between sufficient and necessary conditions can be explored for linear problems using Green functions \cite[Chapter 11, Section 5]{bellotti2023numerical}.
    However, this approach becomes cumbersome and lacks generality, as it requires computing the corresponding Finite Difference scheme for $\conservedVariableDiscrete$.
\end{remark}
\begin{proof}[Proof of \Cref{prop:compactInvariantSets}]
    The proof proceeds by induction over $\indexTime$.
    The base case for $\indexTime = 0$ is trivially verified by \eqref{eq:initializationEquilibrium}.
    Assume that \eqref{eq:compactInvariantSets} holds for $\indexTime \in \naturals$.
    For the equilibrium is an eigenstate---with eigenvalue equal to one---of the relaxation operator (\idEst{} $\vectorial{\collisionOperator}(\distributionFunctionLetter_1^{\atEquilibrium}(\conservedVariableDiscrete), \dots, \distributionFunctionLetter_{\numberVelocities}^{\atEquilibrium}(\conservedVariableDiscrete)) = \transpose{(\distributionFunctionLetter_1^{\atEquilibrium}(\conservedVariableDiscrete), \dots, \distributionFunctionLetter_{\numberVelocities}^{\atEquilibrium}(\conservedVariableDiscrete))}$), we have 
    \begin{multline*}
        \distributionFunction_{\indexVelocity, \vectorial{\indexSpace}}^{\indexTime + 1} - \minimumDistribution_{\indexVelocity} = \distributionFunction_{\indexVelocity, \vectorial{\indexSpace} - \vectorial{\discreteVelocityLetter}_{\indexVelocity}/\latticeVelocity}^{\indexTime, \collided} - \minimumDistribution_{\indexVelocity} = \collisionOperator_{\indexVelocity}(\distributionFunction_{1, \vectorial{\indexSpace} - \vectorial{\discreteVelocityLetter}_{\indexVelocity}/\latticeVelocity}^{\indexTime}, \dots, \distributionFunction_{\numberVelocities, \vectorial{\indexSpace} - \vectorial{\discreteVelocityLetter}_{\indexVelocity}/\latticeVelocity}^{\indexTime}) - \collisionOperator_{\indexVelocity}(\minimumDistribution_1, \dots, \minimumDistribution_{\numberVelocities}) \\
        =\int_0^1 \sum_{p = 1}^{\numberVelocities} \partial_{\distributionFunction_{p}} \collisionOperator_{\indexVelocity}(\vartheta \distributionFunction_{1, \vectorial{\indexSpace} - \vectorial{\discreteVelocityLetter}_{\indexVelocity}/\latticeVelocity}^{\indexTime} + (1-\vartheta)\minimumDistribution_1, \dots,  \vartheta \distributionFunction_{\numberVelocities, \vectorial{\indexSpace} - \vectorial{\discreteVelocityLetter}_{\indexVelocity}/\latticeVelocity}^{\indexTime} + (1-\vartheta)\minimumDistribution_{\numberVelocities})( \distributionFunction_{p, \vectorial{\indexSpace} - \vectorial{\discreteVelocityLetter}_{\indexVelocity}/\latticeVelocity}^{\indexTime} - \minimumDistribution_{p}) \differential\vartheta,
    \end{multline*}
    where $\minimumDistribution_{\indexVelocity}\definitionEquality\distributionFunctionLetter_{\indexVelocity}^{\atEquilibrium}(-\maximumInitialDatum)$, and the last equality holds thanks to the mean value theorem.
    By induction assumption, both $(\distributionFunction_{1, \vectorial{\indexSpace} - \vectorial{\discreteVelocityLetter}_{\indexVelocity}/\latticeVelocity}^{\indexTime}, \dots, \distributionFunction_{\numberVelocities, \vectorial{\indexSpace} - \vectorial{\discreteVelocityLetter}_{\indexVelocity}/\latticeVelocity}^{\indexTime})\in \invariantCompactSetDistributions$ and $(\minimumDistribution_1, \dots, \minimumDistribution_{\numberVelocities})\in\invariantCompactSetDistributions$.
    We have that  $ \distributionFunction_{p, \vectorial{\indexSpace} - \vectorial{\discreteVelocityLetter}_{\indexVelocity}/\latticeVelocity}^{\indexTime} - \minimumDistribution_{p} \geq 0$, and, by monotonicity of the relaxation operator, $\partial_{\distributionFunction_{p}} \collisionOperator_{\indexVelocity}(\vartheta \distributionFunction_{1, \vectorial{\indexSpace} - \vectorial{\discreteVelocityLetter}_{\indexVelocity}/\latticeVelocity}^{\indexTime} + (1-\vartheta)\minimumDistribution_1, \dots,  \vartheta \distributionFunction_{\numberVelocities, \vectorial{\indexSpace} - \vectorial{\discreteVelocityLetter}_{\indexVelocity}/\latticeVelocity}^{\indexTime} + (1-\vartheta)\minimumDistribution_{\numberVelocities})\geq 0$ for all $\vartheta\in[0, 1]$. We deduce that $\distributionFunction_{\indexVelocity, \vectorial{\indexSpace}}^{\indexTime + 1} \geq \minimumDistribution_{\indexVelocity}$. 
    The upper bound is found analogously.
\end{proof}

\subsubsection{$\lebesgueSpace{1}$-contractivity}

To show that the numerical scheme generates a $\lebesgueSpace{1}$-contraction, we first establish the $\petitLebesgueSpace{1}$-contractivity of the relaxation operator.
This furthermore justifies why this is the ``right'' norm for $\reals^{\numberVelocities}$, \confer{} \Cref{rem:normVectorSpace}.
\begin{proposition}[$\petitLebesgueSpace{1}$-contractivity of the relaxation]\label{prop:contractivityCollision}
    Let the conditions by \Cref{prop:monotonicityConditions} be satisfied, and let $(\discrete{g}_1, \dots, \discrete{g}_{\numberVelocities}) \in \invariantCompactSetDistributions$ and $(\distributionFunction_1, \dots, \distributionFunction_{\numberVelocities}) \in \invariantCompactSetDistributions$, then 
    \begin{equation*}
        \sum_{\indexVelocity=1}^{\numberVelocities} |\collisionOperator_{\indexVelocity}(\discrete{g}_1, \dots, \discrete{g}_{\numberVelocities}) -  \collisionOperator_{\indexVelocity}(\distributionFunction_1, \dots, \distributionFunction_{\numberVelocities}) | \leq \sum_{\indexVelocity=1}^{\numberVelocities}|\discrete{g}_{\indexVelocity}-\distributionFunction_{\indexVelocity}|.
    \end{equation*}
\end{proposition}
\begin{proof}
    Let $\discrete{v} \definitionEquality\sum_{\indexVelocity=1}^{\indexVelocity=\numberVelocities}\discrete{g}_{\indexVelocity}$ and $\conservedVariableDiscrete \definitionEquality\sum_{\indexVelocity=1}^{\indexVelocity=\numberVelocities}\distributionFunction_{\indexVelocity}$.
    We rely on the application of the triangle inequality: a careful ``link-wise'' decomposition of the relaxation operator is thus needed.
    It reads as follows.
    \begin{equation*}
        \collisionOperator_{1}(\discrete{g}_1, \dots, \discrete{g}_{\numberVelocities}) -  \collisionOperator_{1}(\distributionFunction_1, \dots, \distributionFunction_{\numberVelocities}) = (1-\relaxationParameterSymmetric + \relaxationParameterSymmetric\equilibriumCoefficientLinear_1) (\discrete{g}_1-\distributionFunction_1) + \relaxationParameterSymmetric\equilibriumCoefficientLinear_1 \sum_{\indexVelocity=2}^{\numberVelocities} (\discrete{g}_{\indexVelocity} - \distributionFunction_{\indexVelocity}).
    \end{equation*}
    For the $\indexLink$-th link, with $\indexLink\in\integerInterval{1}{\numberLinks}$:
    \begin{alignat*}{1}
        &\collisionOperator_{2\indexLink}(\discrete{g}_1, \dots) -  \collisionOperator_{2\indexLink}(\distributionFunction_1, \dots) = \bigl ( 1 - \tfrac{1}{2}(\relaxationParameterSymmetric+\relaxationParameterAntiSymmetric)\bigr )(\discrete{g}_{2\indexLink}-\distributionFunction_{2\indexLink}) \\
        &\qquad- \tfrac{1}{2}(\relaxationParameterSymmetric-\relaxationParameterAntiSymmetric) (\discrete{g}_{2\indexLink+1}-\distributionFunction_{2\indexLink+1})
        +\relaxationParameterSymmetric\equilibriumCoefficientLinear_{2\indexLink}(\discrete{v}-\conservedVariableDiscrete) + \relaxationParameterAntiSymmetric \sum_{\indexDirection=1}^{\spatialDimensionality} \equilibriumCoefficientFlux_{2\indexLink, \indexDirection} (\flux_{\indexDirection}(\discrete{v}) - \flux_{\indexDirection}(\conservedVariableDiscrete)), \\
        &\collisionOperator_{2\indexLink + 1}(\discrete{g}_1, \dots) -  \collisionOperator_{2\indexLink + 1}(\distributionFunction_1, \dots) = \bigl ( 1 - \tfrac{1}{2}(\relaxationParameterSymmetric+\relaxationParameterAntiSymmetric)\bigr )(\discrete{g}_{2\indexLink+1}-\distributionFunction_{2\indexLink + 1}) \\
        &\qquad- \tfrac{1}{2}(\relaxationParameterSymmetric-\relaxationParameterAntiSymmetric) (\discrete{g}_{2\indexLink}-\distributionFunction_{2\indexLink})
        +\relaxationParameterSymmetric\equilibriumCoefficientLinear_{2\indexLink}(\discrete{v}-\conservedVariableDiscrete) - \relaxationParameterAntiSymmetric \sum_{\indexDirection=1}^{\spatialDimensionality} \equilibriumCoefficientFlux_{2\indexLink, \indexDirection} (\flux_{\indexDirection}(\discrete{v}) - \flux_{\indexDirection}(\conservedVariableDiscrete)).
    \end{alignat*}
    Since $\discrete{v}, \conservedVariableDiscrete\in [-\maximumInitialDatum, \maximumInitialDatum]$, we apply the mean value theorem and then decompose $\discrete{v}$ and $\conservedVariableDiscrete$ as sums of distribution functions.
    Introducing $\discrete{w}_{\vartheta}\definitionEquality\vartheta\discrete{v}+(1-\vartheta)\conservedVariableDiscrete\in[-\maximumInitialDatum, \maximumInitialDatum]$:
    \begin{multline*}
        \collisionOperator_{2\indexLink}(\discrete{g}_1, \dots) -  \collisionOperator_{2\indexLink}(\distributionFunction_1, \dots) \\
        = (\discrete{g}_{2\indexLink}-\distributionFunction_{2\indexLink}) \int_0^1 \Bigl ( 1 - \tfrac{1}{2}(\relaxationParameterSymmetric+\relaxationParameterAntiSymmetric)  +\relaxationParameterSymmetric\equilibriumCoefficientLinear_{2\indexLink} + \relaxationParameterAntiSymmetric \sum_{\indexDirection=1}^{\spatialDimensionality} \equilibriumCoefficientFlux_{2\indexLink, \indexDirection} \flux'_{\indexDirection}(\discrete{w}_{\vartheta}) \Bigr )\differential\vartheta \\
        +(\discrete{g}_{2\indexLink+1}-\distributionFunction_{2\indexLink+1}) \int_0^1\Bigl ( - \tfrac{1}{2}(\relaxationParameterSymmetric-\relaxationParameterAntiSymmetric) +\relaxationParameterSymmetric\equilibriumCoefficientLinear_{2\indexLink} + \relaxationParameterAntiSymmetric \sum_{\indexDirection=1}^{\spatialDimensionality} \equilibriumCoefficientFlux_{2\indexLink, \indexDirection} \flux'_{\indexDirection}(\discrete{w}_{\vartheta}) \Bigr )\differential\vartheta \\
        + \sum_{\indexVelocity\neq2\indexLink, 2\indexLink+1}(\discrete{g}_{\indexVelocity}-\distributionFunction_{\indexVelocity}) \int_0^1\Bigl ( \relaxationParameterSymmetric\equilibriumCoefficientLinear_{2\indexLink} + \relaxationParameterAntiSymmetric \sum_{\indexDirection=1}^{\spatialDimensionality} \equilibriumCoefficientFlux_{2\indexLink, \indexDirection} \flux'_{\indexDirection}(\discrete{w}_{\vartheta}) \Bigr )\differential\vartheta, 
    \end{multline*}
    \begin{multline*}
        \collisionOperator_{2\indexLink+1}(\discrete{g}_1, \dots) -  \collisionOperator_{2\indexLink + 1}(\distributionFunction_1, \dots) \\
        = (\discrete{g}_{2\indexLink+1}-\distributionFunction_{2\indexLink+1}) \int_0^1 \Bigl ( 1 - \tfrac{1}{2}(\relaxationParameterSymmetric+\relaxationParameterAntiSymmetric)  +\relaxationParameterSymmetric\equilibriumCoefficientLinear_{2\indexLink} - \relaxationParameterAntiSymmetric \sum_{\indexDirection=1}^{\spatialDimensionality} \equilibriumCoefficientFlux_{2\indexLink, \indexDirection} \flux'_{\indexDirection}(\discrete{w}_{\vartheta}) \Bigr )\differential\vartheta \\
        +(\discrete{g}_{2\indexLink}-\distributionFunction_{2\indexLink}) \int_0^1\Bigl ( - \tfrac{1}{2}(\relaxationParameterSymmetric-\relaxationParameterAntiSymmetric) +\relaxationParameterSymmetric\equilibriumCoefficientLinear_{2\indexLink} - \relaxationParameterAntiSymmetric \sum_{\indexDirection=1}^{\spatialDimensionality} \equilibriumCoefficientFlux_{2\indexLink, \indexDirection} \flux'_{\indexDirection}(\discrete{w}_{\vartheta}) \Bigr )\differential\vartheta \\
        + \sum_{\indexVelocity\neq2\indexLink, 2\indexLink+1}(\discrete{g}_{\indexVelocity}-\distributionFunction_{\indexVelocity}) \int_0^1\Bigl ( \relaxationParameterSymmetric\equilibriumCoefficientLinear_{2\indexLink} - \relaxationParameterAntiSymmetric \sum_{\indexDirection=1}^{\spatialDimensionality} \equilibriumCoefficientFlux_{2\indexLink, \indexDirection} \flux'_{\indexDirection}(\discrete{w}_{\vartheta}) \Bigr )\differential\vartheta.
    \end{multline*}
    The integrands in the previous equations are the entries of the Jacobian of the relaxation operator.
    Therefore, since $\discrete{w}_{\vartheta}\in[-\maximumInitialDatum, \maximumInitialDatum]$, they are non-negative.
    Using this, plus properties of the Lebesgue integral, and the triangle inequality yields
    \begin{multline*}
        |\collisionOperator_{2\indexLink}(\discrete{g}_1, \dots) -  \collisionOperator_{2\indexLink}(\distributionFunction_1, \dots) | + |\collisionOperator_{2\indexLink+1}(\discrete{g}_1, \dots) -  \collisionOperator_{2\indexLink+1}(\distributionFunction_1, \dots) | \\
        \leq ( 1 - \relaxationParameterSymmetric +2 \relaxationParameterSymmetric\equilibriumCoefficientLinear_{2\indexLink}) |\discrete{g}_{2\indexLink}-\distributionFunction_{2\indexLink}|  + ( 1 - \relaxationParameterSymmetric +2 \relaxationParameterSymmetric\equilibriumCoefficientLinear_{2\indexLink}) |\discrete{g}_{2\indexLink + 1}-\distributionFunction_{2\indexLink + 1}| + 2\relaxationParameterSymmetric\equilibriumCoefficientLinear_{2\indexLink} \sum_{\indexVelocity\neq2\indexLink, 2\indexLink+1}|\discrete{g}_{\indexVelocity}-\distributionFunction_{\indexVelocity}|.
    \end{multline*}
    We therefore obtain
    \begin{multline*}
        \sum_{\indexVelocity=1}^{\numberVelocities} |\collisionOperator_{\indexVelocity}(\discrete{g}_1, \dots) -  \collisionOperator_{\indexVelocity}(\distributionFunction_1, \dots) | 
        \leq (1-\relaxationParameterSymmetric + \relaxationParameterSymmetric\equilibriumCoefficientLinear_1) |\discrete{g}_1-\distributionFunction_1| + \relaxationParameterSymmetric\equilibriumCoefficientLinear_1 \sum_{\indexVelocity=2}^{\numberVelocities} |\discrete{g}_{\indexVelocity} - \distributionFunction_{\indexVelocity}| \\
        +\sum_{\indexLink=1}^{\numberLinks}\Bigl ( ( 1 - \relaxationParameterSymmetric +2 \relaxationParameterSymmetric\equilibriumCoefficientLinear_{2\indexLink}) |\discrete{g}_{2\indexLink}-\distributionFunction_{2\indexLink}|  + ( 1 - \relaxationParameterSymmetric +2 \relaxationParameterSymmetric\equilibriumCoefficientLinear_{2\indexLink}) |\discrete{g}_{2\indexLink + 1}-\distributionFunction_{2\indexLink + 1}| \\
        + 2\relaxationParameterSymmetric\equilibriumCoefficientLinear_{2\indexLink} \sum_{\indexVelocity\neq2\indexLink, 2\indexLink+1}|\discrete{g}_{\indexVelocity}-\distributionFunction_{\indexVelocity}|\Bigr )\\
        =\Bigl ( 1-\relaxationParameterSymmetric+ \relaxationParameterSymmetric(\equilibriumCoefficientLinear_1+2\sum_{\indexLink=1}^{\numberLinks}\equilibriumCoefficientLinear_{2\indexLink})\Bigr ) \sum_{\indexVelocity=1}^{\numberVelocities} |\discrete{g}_{\indexVelocity} - \distributionFunction_{\indexVelocity}| = \sum_{\indexVelocity=1}^{\numberVelocities} |\discrete{g}_{\indexVelocity} - \distributionFunction_{\indexVelocity}|,
    \end{multline*}
    where the last equality comes from \eqref{eq:constraintsConsistency3}.
\end{proof}

\begin{proposition}[$\lebesgueSpace{1}$-contractivity of the scheme]
    Let the conditions by \Cref{prop:monotonicityConditions} be satisfied.
    Consider two numerical solutions $\distributionFunction_{\indexVelocity, \vectorial{\indexSpace}}^{\indexTime}$ and $\discrete{g}_{\indexVelocity, \vectorial{\indexSpace}}^{\indexTime}$ obtained from initial data $\conservedVariable^{\initial}$ and $v^{\initial}$, such that $\conservedVariable^{\initial}, v^{\initial}\in \lebesgueSpace{1}(\reals^{\spatialDimensionality}) \cap \lebesgueSpace{\infty}(\reals^{\spatialDimensionality}) \cap \boundedVariationSpace(\reals^{\spatialDimensionality})$ and $\lVert\conservedVariable^{\initial}\rVert_{\lebesgueSpace{\infty}} \leq \maximumInitialDatum$, $\lVert v^{\initial}\rVert_{\lebesgueSpace{\infty}} \leq \maximumInitialDatum$.
    Then, for all $\indexTime \in\naturals$
    \begin{equation}\label{eq:L1contraction}
        \lVert \distributionFunctionsAsFunctionSpaceAnyLetter{g}{\indexTime + 1} - \distributionFunctionsAsFunctionSpace{\indexTime + 1} \rVert_{\lebesgueSpace{1}} \leq \lVert \distributionFunctionsAsFunctionSpaceAnyLetter{g}{\indexTime} - \distributionFunctionsAsFunctionSpace{\indexTime} \rVert_{\lebesgueSpace{1}} \leq \lVert v^{\initial}-\conservedVariable^{\initial}\rVert_{\lebesgueSpace{1}}.
    \end{equation}
    Moreover, there exists $C>0$ such that 
    \begin{equation}\label{eq:timeEquicontinuity1}
        \lVert\distributionFunctionsAsFunctionSpace{\indexTime + 1} - \distributionFunctionsAsFunctionSpace{\indexTime} \rVert_{\lebesgueSpace{1}}  \leq C \spaceStep \totalVariation{\conservedVariable^{\initial}}.
    \end{equation}
\end{proposition}
\begin{proof}
    We can apply \Cref{prop:compactInvariantSets} to both $\distributionFunction_{\indexVelocity, \vectorial{\indexSpace}}^{\indexTime}$ and $\discrete{g}_{\indexVelocity, \vectorial{\indexSpace}}^{\indexTime}$, hence make use of \Cref{prop:contractivityCollision}.
    This yields 
    \begin{align*}
        \lVert \distributionFunctionsAsFunctionSpaceAnyLetter{g}{\indexTime + 1} - \distributionFunctionsAsFunctionSpace{\indexTime + 1} \rVert_{\lebesgueSpace{1}} &= \spaceStep^{\spatialDimensionality} \sum_{\vectorial{\indexSpace}\in\relatives^{\spatialDimensionality}} \sum_{\indexVelocity=1}^{\numberVelocities} |\discrete{g}^{\indexTime, \collided}_{\indexVelocity, \vectorial{\indexSpace}-\vectorial{\discreteVelocityLetter}_{\indexVelocity}/\latticeVelocity} - \distributionFunction^{\indexTime, \collided}_{\indexVelocity, \vectorial{\indexSpace}-\vectorial{\discreteVelocityLetter}_{\indexVelocity}/\latticeVelocity}| = \spaceStep^{\spatialDimensionality} \sum_{\vectorial{\indexSpace}\in\relatives^{\spatialDimensionality}} \sum_{\indexVelocity=1}^{\numberVelocities} |\discrete{g}^{\indexTime, \collided}_{\indexVelocity, \vectorial{\indexSpace}} - \distributionFunction^{\indexTime, \collided}_{\indexVelocity, \vectorial{\indexSpace}}|\\
        &= \spaceStep^{\spatialDimensionality} \sum_{\vectorial{\indexSpace}\in\relatives^{\spatialDimensionality}} \sum_{\indexVelocity=1}^{\numberVelocities} |\collisionOperator_{\indexVelocity}(\discrete{g}^{\indexTime}_{1, \vectorial{\indexSpace}}, \dots, \discrete{g}^{\indexTime}_{\numberVelocities, \vectorial{\indexSpace}}  )-\collisionOperator_{\indexVelocity}(\distributionFunction^{\indexTime}_{1, \vectorial{\indexSpace}}, \dots, \distributionFunction^{\indexTime}_{\numberVelocities, \vectorial{\indexSpace}})|\\
        &\leq \spaceStep^{\spatialDimensionality} \sum_{\vectorial{\indexSpace}\in\relatives^{\spatialDimensionality}} \sum_{\indexVelocity=1}^{\numberVelocities} |\discrete{g}^{\indexTime}_{\indexVelocity, \vectorial{\indexSpace}} - \distributionFunction^{\indexTime}_{\indexVelocity, \vectorial{\indexSpace}}| = \lVert \distributionFunctionsAsFunctionSpaceAnyLetter{g}{\indexTime} - \distributionFunctionsAsFunctionSpace{\indexTime} \rVert_{\lebesgueSpace{1}}.
    \end{align*}
    Iterating on $\indexTime$, we climb time back until reaching 
    \begin{equation*}
        \lVert \distributionFunctionsAsFunctionSpaceAnyLetter{g}{0} - \distributionFunctionsAsFunctionSpace{0} \rVert_{\lebesgueSpace{1}} = \spaceStep^{\spatialDimensionality} \sum_{\vectorial{\indexSpace}\in\relatives^{\spatialDimensionality}} \sum_{\indexVelocity=1}^{\numberVelocities} \Bigl |\distributionFunctionLetter_{\indexVelocity}^{\atEquilibrium}\Bigl (\dashint_{\cell{\vectorial{\indexSpace}}} v^{\initial}(\vectorial{\spaceVariable})\differential\vectorial{\spaceVariable}\Bigr )  -\distributionFunctionLetter_{\indexVelocity}^{\atEquilibrium}\Bigl (\dashint_{\cell{\vectorial{\indexSpace}}} \conservedVariable^{\initial}(\vectorial{\spaceVariable})\differential\vectorial{\spaceVariable}\Bigr ) \Bigr |.
    \end{equation*}
    The equilibria are monotone non-decreasing, thanks to \Cref{prop:monotonicityEquilibria}, thus:
    \begin{equation*}
        \lVert \distributionFunctionsAsFunctionSpaceAnyLetter{g}{0} - \distributionFunctionsAsFunctionSpace{0} \rVert_{\lebesgueSpace{1}} =  \spaceStep^{\spatialDimensionality} \sum_{\vectorial{\indexSpace}\in\relatives^{\spatialDimensionality}} \Bigl |\sum_{\indexVelocity=1}^{\numberVelocities} \distributionFunctionLetter_{\indexVelocity}^{\atEquilibrium}\Bigl (\dashint_{\cell{\vectorial{\indexSpace}}} v^{\initial}(\vectorial{\spaceVariable})\differential\vectorial{\spaceVariable}\Bigr )  -\distributionFunctionLetter_{\indexVelocity}^{\atEquilibrium}\Bigl (\dashint_{\cell{\vectorial{\indexSpace}}} \conservedVariable^{\initial}(\vectorial{\spaceVariable})\differential\vectorial{\spaceVariable}\Bigr ) \Bigr |
        =\lVert v^{\initial}-\conservedVariable^{\initial}\rVert_{\lebesgueSpace{1}},
    \end{equation*}
    using \eqref{eq:constraintsConsistency}, proving \eqref{eq:L1contraction}.
    We analogously prove \eqref{eq:timeEquicontinuity1}:
    \begin{alignat*}{1}
        \lVert &\distributionFunctionsAsFunctionSpace{\indexTime + 1} - \distributionFunctionsAsFunctionSpace{\indexTime}  \rVert_{\lebesgueSpace{1}} = \spaceStep^{\spatialDimensionality} \sum_{\vectorial{\indexSpace}\in\relatives^{\spatialDimensionality}} \sum_{\indexVelocity=1}^{\numberVelocities} |\collisionOperator_{\indexVelocity}(\distributionFunction^{\indexTime}_{1, \vectorial{\indexSpace}-\vectorial{\discreteVelocityLetter}_{\indexVelocity}/\latticeVelocity}, \dots) - \collisionOperator_{\indexVelocity}(\distributionFunction^{\indexTime-1}_{1, \vectorial{\indexSpace}-\vectorial{\discreteVelocityLetter}_{\indexVelocity}/\latticeVelocity}, \dots)|\\
        &\leq\spaceStep^{\spatialDimensionality} \sum_{\vectorial{\indexSpace}\in\relatives^{\spatialDimensionality}} \sum_{\indexVelocity=1}^{\numberVelocities} |\distributionFunction^{\indexTime}_{\indexVelocity, \vectorial{\indexSpace}} - \distributionFunction^{\indexTime-1}_{\indexVelocity, \vectorial{\indexSpace}}| \leq  \spaceStep^{\spatialDimensionality} \sum_{\vectorial{\indexSpace}\in\relatives^{\spatialDimensionality}} \sum_{\indexVelocity=1}^{\numberVelocities} \Bigl |\distributionFunction^{1}_{\indexVelocity, \vectorial{\indexSpace}} - \distributionFunctionLetter_{\indexVelocity}^{\atEquilibrium} \Bigl (\dashint_{\cell{\vectorial{\indexSpace}}}\conservedVariable^{\initial} (\vectorial{\spaceVariable})\differential\vectorial{\spaceVariable}\Bigr ) \Bigr |\\
        &=\spaceStep^{\spatialDimensionality} \sum_{\vectorial{\indexSpace}\in\relatives^{\spatialDimensionality}} \sum_{\indexVelocity=1}^{\numberVelocities} \Bigl |\collisionOperator_{\indexVelocity}\Bigl ( \distributionFunctionLetter_{1}^{\atEquilibrium} \Bigl (\dashint_{\cell{\vectorial{\indexSpace}-\vectorial{\discreteVelocityLetter}_{\indexVelocity}/\latticeVelocity}}\conservedVariable^{\initial} (\vectorial{\spaceVariable})\differential\vectorial{\spaceVariable}\Bigr ), \dots \Bigr )  - \collisionOperator_{\indexVelocity}\Bigl ( \distributionFunctionLetter_{1}^{\atEquilibrium} \Bigl (\dashint_{\cell{\vectorial{\indexSpace}}}\conservedVariable^{\initial} (\vectorial{\spaceVariable})\differential\vectorial{\spaceVariable}\Bigr ), \dots \Bigr ) \Bigr |\\
        &\leq\spaceStep^{\spatialDimensionality} \sum_{\vectorial{\indexSpace}\in\relatives^{\spatialDimensionality}} \sum_{\indexVelocity=1}^{\numberVelocities} \Bigl |\distributionFunctionLetter_{\indexVelocity}^{\atEquilibrium} \Bigl (\dashint_{\cell{\vectorial{\indexSpace}-\vectorial{\discreteVelocityLetter}_{\indexVelocity}/\latticeVelocity}}\conservedVariable^{\initial} (\vectorial{\spaceVariable})\differential\vectorial{\spaceVariable}\Bigr )- \distributionFunctionLetter_{\indexVelocity}^{\atEquilibrium} \Bigl (\dashint_{\cell{\vectorial{\indexSpace}}}\conservedVariable^{\initial} (\vectorial{\spaceVariable})\differential\vectorial{\spaceVariable}\Bigr )\Bigr |\\
        &\leq \sum_{\vectorial{\indexSpace}\in\relatives^{\spatialDimensionality}} \sum_{\indexVelocity=1}^{\numberVelocities}  \Bigl |  \int_{\cell{\vectorial{\indexSpace}-\vectorial{\discreteVelocityLetter}_{\indexVelocity}/\latticeVelocity}}\conservedVariable^{\initial} (\vectorial{\spaceVariable})\differential\vectorial{\spaceVariable} - \int_{\cell{\vectorial{\indexSpace}}}\conservedVariable^{\initial} (\vectorial{\spaceVariable})\differential\vectorial{\spaceVariable} \Bigr | \leq C \spaceStep \totalVariation{\conservedVariable^{\initial}}.
    \end{alignat*}
    Here, to go from the penultimate to the last line, we used that for all $\indexVelocity\in\integerInterval{1}{\numberVelocities}$, if $a, b\in[-\maximumInitialDatum, \maximumInitialDatum]$, then 
    \begin{equation*}
        |\distributionFunctionLetter_{\indexVelocity}^{\atEquilibrium}(a)-\distributionFunctionLetter_{\indexVelocity}^{\atEquilibrium}(b)|\leq \Bigl ( \sup_{\conservedVariableDiscrete\in[-\maximumInitialDatum, \maximumInitialDatum]} \Bigl |\frac{\differential\distributionFunctionLetter_{\indexVelocity}^{\atEquilibrium}(\conservedVariableDiscrete)}{\differential\conservedVariableDiscrete}\Bigr | \Bigr ) |a-b| \leq |a-b|,
    \end{equation*}
    since  \eqref{eq:constraintsConsistency} gives that for $\conservedVariableDiscrete \in [-\maximumInitialDatum, \maximumInitialDatum]$:
    \begin{equation*}
        \sum_{\indexVelocity = 1}^{\numberVelocities} \frac{\differential\distributionFunctionLetter_{\indexVelocity}^{\atEquilibrium}(\conservedVariableDiscrete)}{\differential\conservedVariableDiscrete}  = 1, \qquad \text{and monotonicity gives}\qquad \frac{\differential\distributionFunctionLetter_{\indexVelocity}^{\atEquilibrium}(\conservedVariableDiscrete)}{\differential\conservedVariableDiscrete}\geq 0.
    \end{equation*}
\end{proof}

\begin{corollary}[Equicontinuity in time]
    Let the condition by \Cref{prop:monotonicityConditions} be satisfied.
    Then, there exists $C>0$ such that, for every $0\leq \timeVariable\leq\tilde{\timeVariable}$
    \begin{equation*}
        \lVert \distributionFunctionsAsFunction(\tilde{\timeVariable}, \cdot) - \distributionFunctionsAsFunction(\timeVariable, \cdot) \rVert_{\lebesgueSpace{1}}\leq \latticeVelocity C (\tilde{\timeVariable}-\timeVariable + \timeStep) \totalVariation{\conservedVariable^{\initial}}.
    \end{equation*}
\end{corollary}

\subsubsection{Total variation estimates}

\begin{proposition}
    Let the conditions by \Cref{prop:monotonicityConditions} be satisfied. Then, for all $\indexTime \in\naturals$
    \begin{equation*}
        \totalVariationVectorial{\distributionFunctionsAsFunctionSpace{\indexTime + 1}}\leq  \totalVariationVectorial{\distributionFunctionsAsFunctionSpace{\indexTime}} \leq \dots \leq \totalVariationVectorial{\distributionFunctionsAsFunctionSpace{0}} \leq  \totalVariation{\conservedVariable^{\initial}},
    \end{equation*}
    and 
    \begin{equation}\label{eq:totalVariationConservedControlledByVectorial}
        \totalVariation{\conservedVariableDiscreteAsFunctionSpace{\indexTime}}\leq \totalVariationVectorial{\distributionFunctionsAsFunctionSpace{\indexTime}}.
    \end{equation}
\end{proposition}
\begin{proof}
    We have
    \begin{align*}
        \totalVariationVectorial{\distributionFunctionsAsFunctionSpace{\indexTime + 1}} &= \spaceStep^{\spatialDimensionality-1}\sum_{\vectorial{\indexSpace}\in\relatives^{\spatialDimensionality}}\sum_{\indexDirection=1}^{\spatialDimensionality} \sum_{\indexVelocity=1}^{\numberVelocities}|\distributionFunction_{\indexVelocity, \vectorial{\indexSpace}+\canonicalBasisVector{\indexDirection}}^{\indexTime +1} - \distributionFunction_{\indexVelocity, \vectorial{\indexSpace}}^{\indexTime + 1}| \\
        &= \spaceStep^{\spatialDimensionality-1}\sum_{\vectorial{\indexSpace}\in\relatives^{\spatialDimensionality}}\sum_{\indexDirection=1}^{\spatialDimensionality} \sum_{\indexVelocity=1}^{\numberVelocities}| \collisionOperator_{\indexVelocity}(\distributionFunction_{1, \vectorial{\indexSpace}+\canonicalBasisVector{\indexDirection} - \vectorial{\discreteVelocityLetter}_{\indexVelocity}/\latticeVelocity}^{\indexTime}, \dots)  - \collisionOperator_{\indexVelocity}(\distributionFunction_{1, \vectorial{\indexSpace}- \vectorial{\discreteVelocityLetter}_{\indexVelocity}/\latticeVelocity}^{\indexTime}, \dots)| \\
        &=  \spaceStep^{\spatialDimensionality-1}\sum_{\vectorial{\indexSpace}\in\relatives^{\spatialDimensionality}}\sum_{\indexDirection=1}^{\spatialDimensionality} \sum_{\indexVelocity=1}^{\numberVelocities}| \collisionOperator_{\indexVelocity}(\distributionFunction_{1, \vectorial{\indexSpace}+\canonicalBasisVector{\indexDirection} }^{\indexTime}, \dots, \distributionFunction_{\numberVelocities, \vectorial{\indexSpace}+\canonicalBasisVector{\indexDirection} }^{\indexTime})  - \collisionOperator_{\indexVelocity}(\distributionFunction_{1, \vectorial{\indexSpace}}^{\indexTime}, \dots, \distributionFunction_{\numberVelocities, \vectorial{\indexSpace}}^{\indexTime})|.
    \end{align*}
    As $(\distributionFunction_{1, \vectorial{\indexSpace}+\canonicalBasisVector{\indexDirection} }^{\indexTime}, \dots, \distributionFunction_{\numberVelocities, \vectorial{\indexSpace}+\canonicalBasisVector{\indexDirection} }^{\indexTime}), (\distributionFunction_{1, \vectorial{\indexSpace}}^{\indexTime}, \dots, \distributionFunction_{\numberVelocities, \vectorial{\indexSpace}}^{\indexTime})\in\invariantCompactSetDistributions$, by \Cref{prop:contractivityCollision}:
    \begin{equation*}
        \totalVariationVectorial{\distributionFunctionsAsFunctionSpace{\indexTime + 1}}\leq \spaceStep^{\spatialDimensionality-1}\sum_{\vectorial{\indexSpace}\in\relatives^{\spatialDimensionality}}\sum_{\indexDirection=1}^{\spatialDimensionality} \sum_{\indexVelocity=1}^{\numberVelocities}| \distributionFunction_{\indexVelocity, \vectorial{\indexSpace}+\canonicalBasisVector{\indexDirection} }^{\indexTime} - \distributionFunction_{\indexVelocity, \vectorial{\indexSpace}}^{\indexTime}| = \totalVariationVectorial{\distributionFunctionsAsFunctionSpace{\indexTime}} \leq \dots \leq \totalVariationVectorial{\distributionFunctionsAsFunctionSpace{0}}.
    \end{equation*}
    Then, by monotonicity of the equilibria and \eqref{eq:constraintsConsistency}, we have 
    \begin{align*}
        \totalVariationVectorial{\distributionFunctionsAsFunctionSpace{0}} &= \spaceStep^{\spatialDimensionality-1}\sum_{\vectorial{\indexSpace}\in\relatives^{\spatialDimensionality}}\sum_{\indexDirection=1}^{\spatialDimensionality} \sum_{\indexVelocity=1}^{\numberVelocities}\Bigl |\distributionFunctionLetter_{\indexVelocity}^{\atEquilibrium} \Bigl (\dashint_{\cell{\vectorial{\indexSpace}+\canonicalBasisVector{\indexDirection}}}\conservedVariable^{\initial} (\vectorial{\spaceVariable})\differential\vectorial{\spaceVariable}\Bigr )- \distributionFunctionLetter_{\indexVelocity}^{\atEquilibrium} \Bigl (\dashint_{\cell{\vectorial{\indexSpace}}}\conservedVariable^{\initial} (\vectorial{\spaceVariable})\differential\vectorial{\spaceVariable}\Bigr )\Bigr |\\
        &= \spaceStep^{\spatialDimensionality-1}\sum_{\vectorial{\indexSpace}\in\relatives^{\spatialDimensionality}} \sum_{\indexDirection=1}^{\spatialDimensionality}\Bigl |  \dashint_{\cell{\vectorial{\indexSpace}+\canonicalBasisVector{\indexDirection}}}\conservedVariable^{\initial} (\vectorial{\spaceVariable})\differential\vectorial{\spaceVariable} - \dashint_{\cell{\vectorial{\indexSpace}}}\conservedVariable^{\initial} (\vectorial{\spaceVariable})\differential\vectorial{\spaceVariable} \Bigr | \leq \totalVariation{\conservedVariable^{\initial}}.
    \end{align*}
    To finish, \eqref{eq:totalVariationConservedControlledByVectorial} is a straightforward consequence of the triangle inequality.
\end{proof}

\subsubsection{Convergence to the equilibrium}

We now provide estimates ensuring that at each time-step, the discrete solution remains within $\bigO{\spaceStep}$ to the equilibrium.
\begin{proposition}[Closeness to the equilibrium]\label{prop:convergenceEquilibrium}
    Let the conditions by \Cref{prop:monotonicityConditions} be satisfied.
    Then, there exists $C>0$ such that, for all $\indexTime\in\naturals$
    \begin{equation}\label{eq:convergenceEquilibrium}
        \lVert \distributionFunctionsAsFunctionSpace{\indexTime} - \vectorial{\distributionFunctionLetter}^{\atEquilibrium}(\conservedVariableDiscreteAsFunctionSpace{\indexTime})\rVert_{\lebesgueSpace{1}} \leq \frac{C}{1-\max(|1-\relaxationParameterSymmetric|, |1-\relaxationParameterAntiSymmetric|)}  \spaceStep\totalVariation{\conservedVariable^{\initial}}.
    \end{equation}
\end{proposition}
\begin{proof}
    Let us introduce the shorthand $\delta^{\atEquilibrium, \indexTime}\definitionEquality \lVert \distributionFunctionsAsFunctionSpace{\indexTime} - \vectorial{\distributionFunctionLetter}^{\atEquilibrium}(\conservedVariableDiscreteAsFunctionSpace{\indexTime})\rVert_{\lebesgueSpace{1}}$. We have 
    \begin{align*}
        \delta^{\atEquilibrium, \indexTime + 1} &= \spaceStep^{\spatialDimensionality} \sum_{\vectorial{\indexSpace}\in\relatives^{\spatialDimensionality}} \sum_{\indexVelocity=1}^{\numberVelocities} |\collisionOperator_{\indexVelocity}(\distributionFunction_{1, \vectorial{\indexSpace}-\vectorial{\discreteVelocityLetter}_{\indexVelocity}/\latticeVelocity}^{\indexTime}, \dots, \distributionFunction_{\numberVelocities, \vectorial{\indexSpace}-\vectorial{\discreteVelocityLetter}_{\indexVelocity}/\latticeVelocity}^{\indexTime}) - \distributionFunctionLetter^{\atEquilibrium}_{\indexVelocity}(\conservedVariableDiscrete_{\vectorial{\indexSpace}}^{\indexTime + 1})|\\
        &\leq \spaceStep^{\spatialDimensionality} \sum_{\vectorial{\indexSpace}\in\relatives^{\spatialDimensionality}} \sum_{\indexVelocity=1}^{\numberVelocities} (|\collisionOperator_{\indexVelocity}(\distributionFunction_{1, \vectorial{\indexSpace}-\vectorial{\discreteVelocityLetter}_{\indexVelocity}/\latticeVelocity}^{\indexTime}, \dots) - \distributionFunctionLetter^{\atEquilibrium}_{\indexVelocity}(\conservedVariableDiscrete_{\vectorial{\indexSpace}-\vectorial{\discreteVelocityLetter}_{\indexVelocity}/\latticeVelocity}^{\indexTime})| + | \distributionFunctionLetter^{\atEquilibrium}_{\indexVelocity}(\conservedVariableDiscrete_{\vectorial{\indexSpace}}^{\indexTime + 1}) - \distributionFunctionLetter^{\atEquilibrium}_{\indexVelocity}(\conservedVariableDiscrete_{\vectorial{\indexSpace}-\vectorial{\discreteVelocityLetter}_{\indexVelocity}/\latticeVelocity}^{\indexTime})|)\\
        &\leq \spaceStep^{\spatialDimensionality} \sum_{\vectorial{\indexSpace}\in\relatives^{\spatialDimensionality}} \sum_{\indexVelocity=1}^{\numberVelocities} |\collisionOperator_{\indexVelocity}(\distributionFunction_{1, \vectorial{\indexSpace}}^{\indexTime}, \dots, \distributionFunction_{\numberVelocities, \vectorial{\indexSpace}}^{\indexTime}) - \distributionFunctionLetter^{\atEquilibrium}_{\indexVelocity}(\conservedVariableDiscrete_{\vectorial{\indexSpace}}^{\indexTime})| + \spaceStep^{\spatialDimensionality} \sum_{\vectorial{\indexSpace}\in\relatives^{\spatialDimensionality}} \sum_{\indexVelocity=1}^{\numberVelocities}  |\distributionFunctionLetter^{\atEquilibrium}_{\indexVelocity}(\conservedVariableDiscrete_{\vectorial{\indexSpace}-\vectorial{\discreteVelocityLetter}_{\indexVelocity}/\latticeVelocity}^{\indexTime}) - \distributionFunctionLetter^{\atEquilibrium}_{\indexVelocity}(\conservedVariableDiscrete_{\vectorial{\indexSpace}}^{\indexTime + 1})|.
    \end{align*}
    For the first term 
    \begin{multline*}
        \sum_{\indexVelocity=1}^{\numberVelocities} |\collisionOperator_{\indexVelocity}(\distributionFunction_{1, \vectorial{\indexSpace}}^{\indexTime}, \dots, \distributionFunction_{\numberVelocities, \vectorial{\indexSpace}}^{\indexTime}) - \distributionFunctionLetter^{\atEquilibrium}_{\indexVelocity}(\conservedVariableDiscrete_{\vectorial{\indexSpace}}^{\indexTime})| \leq |1-\relaxationParameterSymmetric||\distributionFunction_{1, \vectorial{\indexSpace}}^{\indexTime}-\distributionFunctionLetter^{\atEquilibrium}_{1}(\conservedVariableDiscrete_{\vectorial{\indexSpace}}^{\indexTime})| \\
        + \sum_{\indexVelocity = 2}^{\numberVelocities}\Bigl ( \Bigl |1 - \tfrac{1}{2}(\relaxationParameterSymmetric+\relaxationParameterAntiSymmetric) \Bigr | + \Bigl |\tfrac{1}{2}(\relaxationParameterSymmetric-\relaxationParameterAntiSymmetric) \Bigr |\Bigr ) |\distributionFunction_{\indexVelocity, \vectorial{\indexSpace}}^{\indexTime} - \distributionFunctionLetter^{\atEquilibrium}_{\indexVelocity}(\conservedVariableDiscrete_{\vectorial{\indexSpace}}^{\indexTime})|.
    \end{multline*}
    One easily sees that $ |1 - \tfrac{1}{2}(\relaxationParameterSymmetric+\relaxationParameterAntiSymmetric)  | + |\tfrac{1}{2}(\relaxationParameterSymmetric-\relaxationParameterAntiSymmetric)  | = \max(|1-\relaxationParameterSymmetric|, |1-\relaxationParameterAntiSymmetric|)$, thus
    \begin{equation*}
        \sum_{\indexVelocity=1}^{\numberVelocities} |\collisionOperator_{\indexVelocity}(\distributionFunction_{1, \vectorial{\indexSpace}}^{\indexTime}, \dots, \distributionFunction_{\numberVelocities, \vectorial{\indexSpace}}^{\indexTime}) - \distributionFunctionLetter^{\atEquilibrium}_{\indexVelocity}(\conservedVariableDiscrete_{\vectorial{\indexSpace}}^{\indexTime})|\leq \max(|1-\relaxationParameterSymmetric|, |1-\relaxationParameterAntiSymmetric|)  \sum_{\indexVelocity = 1}^{\numberVelocities} |\distributionFunction_{\indexVelocity, \vectorial{\indexSpace}}^{\indexTime} - \distributionFunctionLetter^{\atEquilibrium}_{\indexVelocity}(\conservedVariableDiscrete_{\vectorial{\indexSpace}}^{\indexTime})|.
    \end{equation*}
    Therefore, we have shown that $\delta^{\atEquilibrium, \indexTime + 1} \leq \max(|1-\relaxationParameterSymmetric|, |1-\relaxationParameterAntiSymmetric|) \delta^{\atEquilibrium, \indexTime}  + C \spaceStep\totalVariation{\conservedVariable^{\initial}}$.
    Using \Cref{prop:neverEqualToTwo}, we have that $0\leq\max(|1-\relaxationParameterSymmetric|, |1-\relaxationParameterAntiSymmetric|)<1$, thus---considering that $\delta^{\atEquilibrium, 0} = 0$ yields 
    \begin{align}
        \delta^{\atEquilibrium, \indexTime} &\leq C \spaceStep\totalVariation{\conservedVariable^{\initial}} \frac{\max(|1-\relaxationParameterSymmetric|, |1-\relaxationParameterAntiSymmetric|)^{\indexTime}-1}{\max(|1-\relaxationParameterSymmetric|, |1-\relaxationParameterAntiSymmetric|)-1} \label{eq:geometricConvergenceEquilibrium}\\
        &\leq   \frac{C}{1-\max(|1-\relaxationParameterSymmetric|, |1-\relaxationParameterAntiSymmetric|)}  \spaceStep\totalVariation{\conservedVariable^{\initial}}.\nonumber
    \end{align}
\end{proof}

\begin{remark}[On the way of converging to the equilibrium]
    In practice, see \Cref{sec:D1Q3convergenceEquilibrium}, the numerical solution converges to the neighborhood of the equilibrium in a geometric fashion. This is not surprising taking \eqref{eq:geometricConvergenceEquilibrium} into account, since its right-hand side represents a geometrically damped sequence in $\indexTime$, and demonstrates that the bounds \eqref{eq:convergenceEquilibrium} and \eqref{eq:geometricConvergenceEquilibrium} are sharp.
\end{remark}

\subsubsection{Convergence}

\begin{theorem}[Convergence to a weak solution]\label{thm:convergenceWeakSolution}
    Let the conditions by \Cref{prop:monotonicityConditions} be satisfied.
    Let $(\spaceStep_{p})_{p\in\naturals} $ be a sequence of non-negative space-steps such that $\lim_{p\to+\infty}\spaceStep_p = 0$.
    Then, there exists a subsequence of space-steps, also denoted $(\spaceStep_{p})_{p\in\naturals} $ for simplicity, and a function $\vectorial{\limitDistributionFunction}$ such that $\vectorial{\limitDistributionFunction}(\timeVariable, \cdot)\in\lebesgueSpace{1}(\reals^{\spatialDimensionality})$ and $\vectorial{\limitDistributionFunction}(\timeVariable, \vectorial{\spaceVariable})\in\invariantCompactSetDistributions$ a.e. in $\vectorial{\spaceVariable}$, for $\timeVariable\geq 0$ such that, for $\finalTime>0$ 
    \begin{equation}\label{eq:limit}
        \lim_{p\to +\infty} \lVert\distributionFunctionsAsFunctionWithStep{\discreteMark_p} - \vectorial{\limitDistributionFunction} \rVert_{\lebesgueTime{\infty}([0, \finalTime])\lebesgueInSpace{1}} = 0,
    \end{equation}
    and, setting $\limitConservedMoment \definitionEquality \sum_{\indexVelocity=1}^{\indexVelocity=\numberVelocities}\limitDistributionFunction_{\indexVelocity}$, such that $\limitConservedMoment(\timeVariable, \vectorial{\spaceVariable})\in[-\maximumInitialDatum, \maximumInitialDatum]$ a.e. in $\vectorial{\spaceVariable}$, with $\lim_{p\to +\infty} \lVert\conservedVariableDiscreteAsFunctionWithStep{\discreteMark_p} -\limitConservedMoment \rVert_{\lebesgueTime{\infty}([0, \finalTime])\lebesgueInSpace{1}} = 0$.
    Moreover, the limit distribution functions $\vectorial{\limitDistributionFunction}$ are at equilibrium, namely $\vectorial{\limitDistributionFunction}(\timeVariable, \vectorial{\spaceVariable}) = \vectorial{\distributionFunctionLetter}^{\atEquilibrium}(\limitConservedMoment(\timeVariable, \vectorial{\spaceVariable}))$ a.e. in $\vectorial{\spaceVariable}$.
    Finally, $\limitConservedMoment$ is a weak solution of \eqref{eq:conservationLaw}.
\end{theorem}
\begin{proof}
    We first extract a converging subsequence, since all the needed properties used in \cite{crandall1980monotone} are proved, which gives $\lim_{p\to +\infty} \lVert\distributionFunctionsAsFunctionWithStep{\discreteMark_p} - \vectorial{\limitDistributionFunction} \rVert_{\lebesgueTime{\infty}([0, \finalTime])\lebesgueInSpace{1}} = 0$.
    Upon extracting again, we know that $\lebesgueInSpace{1}$-convergence implies point-wise convergence almost everywhere, which entails that $\vectorial{\limitDistributionFunction}(\timeVariable, \vectorial{\spaceVariable}) \in\invariantCompactSetDistributions$, thus $\limitConservedMoment(\timeVariable, \vectorial{\spaceVariable})\in[-\maximumInitialDatum, \maximumInitialDatum]$, a.e. in $\vectorial{\spaceVariable}$.
    With all norms being $\lebesgueTime{\infty}([0, \finalTime])\lebesgueInSpace{1}$, we have 
    \begin{alignat*}{1}
        \lVert\vectorial{\limitDistributionFunction} - \vectorial{\distributionFunctionLetter}^{\atEquilibrium}(\limitConservedMoment)&\rVert_{\lebesgueTime{\infty}([0, \finalTime])\lebesgueInSpace{1}} \\
        \leq &\lVert\vectorial{\limitDistributionFunction} -\distributionFunctionsAsFunctionWithStep{\discreteMark_{p}} \rVert + \lVert\distributionFunctionsAsFunctionWithStep{\discreteMark_{p}} - \vectorial{\distributionFunctionLetter}^{\atEquilibrium}(\conservedVariableDiscreteAsFunctionWithStep{\discreteMark_{p}})\rVert + \lVert \vectorial{\distributionFunctionLetter}^{\atEquilibrium}(\conservedVariableDiscreteAsFunctionWithStep{\discreteMark_{p}}) - \vectorial{\distributionFunctionLetter}^{\atEquilibrium}(\limitConservedMoment)\rVert\\
        \leq&\lVert\vectorial{\limitDistributionFunction} -\distributionFunctionsAsFunctionWithStep{\discreteMark_{p}} \rVert  + C(1-\max(|1-\relaxationParameterSymmetric|, |1-\relaxationParameterAntiSymmetric|))^{-1}  \spaceStep_{p}\totalVariation{\conservedVariable^{\initial}} + \lVert \conservedVariableDiscreteAsFunctionWithStep{\discreteMark_{p}} - \limitConservedMoment\rVert.
    \end{alignat*}
    The last inequality comes from monotonicity of the equilibria, which can be invoked since $\limitConservedMoment(\timeVariable, \vectorial{\spaceVariable})\in[-\maximumInitialDatum, \maximumInitialDatum]$, a.e. in $\vectorial{\spaceVariable}$.
    Letting $p\to+\infty$, we deduce $\lVert\vectorial{\limitDistributionFunction} - \vectorial{\distributionFunctionLetter}^{\atEquilibrium}(\limitConservedMoment)\rVert_{\lebesgueTime{\infty}([0, \finalTime])\lebesgueInSpace{1}} = 0$, hence $\vectorial{\limitDistributionFunction}(\timeVariable, \vectorial{\spaceVariable}) = \vectorial{\distributionFunctionLetter}^{\atEquilibrium}(\limitConservedMoment(\timeVariable, \vectorial{\spaceVariable}))$ a.e. in $\vectorial{\spaceVariable}$.
    To show that $\limitConservedMoment$ is a weak solution of \eqref{eq:conservationLaw}, the fact that $\spaceStep = \spaceStep_p$, $\timeStep=\spaceStep_p / \latticeVelocity$, and that limits are for $p\to +\infty$ is understood.
        Consider a test function $\testFunction\in C_c^1([0, +\infty)\times\reals^{\spatialDimensionality})$, 
        \begin{equation*}
            \testFunctionDiscrete_{\vectorial{\indexSpace}}^{\indexTime} \definitionEquality \testFunction(\timeGridPoint{\indexTime}, \spaceGridPoint{\vectorial{\indexSpace}}),\qquad \text{and}\qquad \testFunctionWithStep(\timeVariable, \vectorial{\spaceVariable}) \definitionEquality \sum_{\indexTime\in\naturals} \sum_{\vectorial{\indexSpace}\in\relatives^{\spatialDimensionality}} \testFunctionDiscrete_{\vectorial{\indexSpace}}^{\indexTime} \indicatorFunction{[\timeGridPoint{\indexTime}, \timeGridPoint{\indexTime+1})} (\timeVariable)\indicatorFunction{\cell{\vectorial{\indexSpace}}} (\vectorial{\spaceVariable}).
        \end{equation*}
        Summing the schemes \eqref{eq:collision}/\eqref{eq:transport}, \idEst{} $\distributionFunction_{\indexVelocity, \vectorial{\indexSpace}}^{\indexTime+1} = \collisionOperator_{\indexVelocity}(\distributionFunction_{1, \vectorial{\indexSpace}-\vectorial{\discreteVelocityLetter}_{\indexVelocity}/\latticeVelocity}^{\indexTime}, \dots, \distributionFunction_{\numberVelocities, \vectorial{\indexSpace}-\vectorial{\discreteVelocityLetter}_{\indexVelocity}/\latticeVelocity}^{\indexTime})$, over $\indexVelocity\in\integerInterval{1}{\numberVelocities}$, multiplying by the test function and summing in time and space, we obtain 
        \begin{equation*}
            \timeStep\spaceStep^{\spatialDimensionality}\sum_{\indexTime\in\naturals} \sum_{\vectorial{\indexSpace}\in\relatives^{\spatialDimensionality}} \frac{\conservedVariableDiscrete_{\vectorial{\indexSpace}}^{\indexTime + 1} - \conservedVariableDiscrete_{\vectorial{\indexSpace}}^{\indexTime}}{\timeStep} \testFunctionDiscrete_{\vectorial{\indexSpace}}^{\indexTime} 
    = 
    \timeStep\spaceStep^{\spatialDimensionality}\sum_{\indexTime\in\naturals} \sum_{\vectorial{\indexSpace}\in\relatives^{\spatialDimensionality}} \frac{\sum\limits_{\indexVelocity=1}^{\numberVelocities}\collisionOperator_{\indexVelocity}(\distributionFunction_{1, \vectorial{\indexSpace}-\vectorial{\discreteVelocityLetter}_{\indexVelocity}/\latticeVelocity}^{\indexTime}, \dots) - \conservedVariableDiscrete_{\vectorial{\indexSpace}}^{\indexTime}}{\timeStep}\testFunctionDiscrete_{\vectorial{\indexSpace}}^{\indexTime} .
        \end{equation*}
Standard summations-by-parts and the fact that test functions are compactly supported in space give, switching to integrals: 
\begin{multline}\label{eq:discreteIPPIntegrals}
    \overbrace{\int_{\timeStep}^{+\infty}\hspace{-0.5em}\int_{\reals^{\spatialDimensionality}}\conservedVariableDiscreteAsFunction(\timeVariable, \vectorial{\spaceVariable})\frac{\testFunctionWithStep(\timeVariable - \timeStep, \vectorial{\spaceVariable})-\testFunctionWithStep(\timeVariable, \vectorial{\spaceVariable})}{\timeStep}\differential{\vectorial{\spaceVariable}}\differential{\timeVariable}}^{D} - \overbrace{\int_{\reals^{\spatialDimensionality}}\conservedVariableDiscreteAsFunction(0, \vectorial{\spaceVariable})\testFunctionWithStep(0, \vectorial{\spaceVariable})\differential{\vectorial{\spaceVariable}}}^{I}\\
    =\underbrace{\int_{0}^{+\infty}\hspace{-0.5em}\int_{\reals^{\spatialDimensionality}}\frac{1}{\timeStep}\Bigl ( \sum\limits_{\indexVelocity=1}^{\numberVelocities}\collisionOperator_{\indexVelocity}({\distributionFunctionsAsFunction(\timeVariable, \vectorial{\spaceVariable})}) \testFunctionWithStep(\timeVariable, \vectorial{\spaceVariable}+\spaceStep\vectorial{\discreteVelocityLetter}_{\indexVelocity}/\latticeVelocity) - \conservedVariableDiscreteAsFunction(\timeVariable, \vectorial{\spaceVariable})\testFunctionWithStep(\timeVariable, \vectorial{\spaceVariable})\Bigr ) \differential{\vectorial{\spaceVariable}}\differential{\timeVariable}}_{F}.
\end{multline}
    The terms $D$ (\idEst{} ``derivative'') and $I$ (\idEst{} ``initial'') are precisely the same as in the proof of the Lax-Wendroff theorem on \cite[Page 100]{godlewski1991hyperbolic}, where it is shown that 
    \begin{equation*}
        D\to -\int_{0}^{+\infty}\hspace{-0.5em}\int_{\reals^{\spatialDimensionality}}\limitConservedMoment(\timeVariable, \vectorial{\spaceVariable})\partial_{\timeVariable}\testFunction(\timeVariable, \vectorial{\spaceVariable})\differential{\vectorial{\spaceVariable}}\differential{\timeVariable}\qquad \text{and}\qquad I\to \int_{\reals^{\spatialDimensionality}}\conservedVariable^{\initial}(\vectorial{\spaceVariable})\testFunction(0, \vectorial{\spaceVariable})\differential{\vectorial{\spaceVariable}}.
    \end{equation*}
    We are left to handle $F$ (\idEst{} ``flux''), which is different from standard Finite Volume schemes but can be treated by analogous arguments---see supplementary material.
    We obtain 
    \begin{multline*}
        F\to \int_{0}^{+\infty}\hspace{-0.5em}\int_{\reals^{\spatialDimensionality}} \sum_{\indexVelocity=1}^{\numberVelocities}\distributionFunctionLetter_{\indexVelocity}^{\atEquilibrium}(\limitConservedMoment(\timeVariable, \vectorial{\spaceVariable}))\sum_{\indexDirection = 1}^{\spatialDimensionality} \discreteVelocityLetter_{\indexVelocity, \indexDirection}\partial_{\spaceVariable_{\indexDirection}}\testFunction(\timeVariable, \vectorial{\spaceVariable})\differential{\vectorial{\spaceVariable}}\differential{\timeVariable}\\
         = \int_{0}^{+\infty}\hspace{-0.5em}\int_{\reals^{\spatialDimensionality}}\sum_{\indexDirection = 1}^{\spatialDimensionality} \sum_{\indexLink=1}^{\numberLinks} \discreteVelocityLetter_{2\indexLink, \indexDirection} (\distributionFunctionLetter_{2\indexLink}^{\atEquilibrium}(\limitConservedMoment(\timeVariable, \vectorial{\spaceVariable})) - \distributionFunctionLetter_{2\indexLink + 1}^{\atEquilibrium}(\limitConservedMoment(\timeVariable, \vectorial{\spaceVariable}))) \partial_{\spaceVariable_{\indexDirection}}\testFunction(\timeVariable, \vectorial{\spaceVariable})\differential{\vectorial{\spaceVariable}}\differential{\timeVariable} \\
         = \int_{0}^{+\infty}\hspace{-0.5em}\int_{\reals^{\spatialDimensionality}}\sum_{\indexDirection = 1}^{\spatialDimensionality} \flux_{\indexDirection}(\limitConservedMoment(\timeVariable, \vectorial{\spaceVariable})) \partial_{\spaceVariable_{\indexDirection}}\testFunction(\timeVariable, \vectorial{\spaceVariable})\differential{\vectorial{\spaceVariable}}\differential{\timeVariable},
    \end{multline*}
    where the first equality comes from the particular choice of discrete velocities \eqref{eq:choiceDiscreteVelocities}, and the second one from the consistency constraints \eqref{eq:constraintsConsistency}.
    This shows that the limit equation is the weak form of \eqref{eq:conservationLaw}, with solution $\limitConservedMoment$.
\end{proof}

\subsection{Convergence to the weak entropy solution}

\begin{theorem}
    Under the same assumptions as \Cref{thm:convergenceWeakSolution}, the limit $\limitConservedMoment$ is the unique weak entropy solution to \eqref{eq:conservationLaw}.
\end{theorem}
\begin{proof}
    We utilize Krushkov kinetic entropies as in \cite{natalini1998discrete}: $\kineticEntropy_{\indexVelocity}(\distributionFunction_{\indexVelocity}) \definitionEquality |\distributionFunction_{\indexVelocity} - \distributionFunctionLetter_{\indexVelocity}^{\atEquilibrium}(\krushkovParameter)|$, for $\krushkovParameter\in\reals$, with $\indexVelocity\in\integerInterval{1}{\numberVelocities}$.
    We study the case $\krushkovParameter\in[-\maximumInitialDatum, \maximumInitialDatum]$, since weak consistency, \confer{} \Cref{thm:convergenceWeakSolution}, rules the case $\krushkovParameter\in\reals\smallsetminus[-\maximumInitialDatum, \maximumInitialDatum]$.
    We follow the approach by \cite{caetano2024result}: the discrete entropy balance is written on post-relaxation quantities, as 
    \begin{align*}
        \sum\limits_{\indexVelocity=1}^{\numberVelocities} \kineticEntropy_{\indexVelocity}(\distributionFunction_{\indexVelocity, \vectorial{\indexSpace}}^{\indexTime + 1, \collided}) &= \sum_{\indexVelocity=1}^{\numberVelocities}  |\collisionOperator_{\indexVelocity}(\distributionFunction_{1, \vectorial{\indexSpace}}^{\indexTime, \collided}, \distributionFunction_{2, \vectorial{\indexSpace} - \vectorial{\discreteVelocityLetter}_2/\latticeVelocity}^{\indexTime, \collided}, \dots, \distributionFunction_{\numberVelocities, \vectorial{\indexSpace} - \vectorial{\discreteVelocityLetter}_{\numberVelocities}/\latticeVelocity}^{\indexTime, \collided}) - \collisionOperator_{\indexVelocity}(\distributionFunctionLetter_1^{\atEquilibrium}(\krushkovParameter), \dots, \distributionFunctionLetter_{\numberVelocities}^{\atEquilibrium}(\krushkovParameter))|\\
        &\leq\sum_{\indexVelocity=1}^{\numberVelocities}  |\distributionFunction_{\indexVelocity, \vectorial{\indexSpace} - \vectorial{\discreteVelocityLetter}_{\indexVelocity}/\latticeVelocity}^{\indexTime, \collided} - \distributionFunctionLetter_{\indexVelocity}^{\atEquilibrium}(\krushkovParameter)| = \sum\limits_{\indexVelocity=1}^{\numberVelocities} \kineticEntropy_{\indexVelocity}(\distributionFunction_{\indexVelocity, \vectorial{\indexSpace} - \vectorial{\discreteVelocityLetter}_{\indexVelocity}/\latticeVelocity}^{\indexTime, \collided}),
    \end{align*}
    where we have used \Cref{prop:contractivityCollision}.
    We obtain 
    \begin{equation*}
        \frac{\sum\limits_{\indexVelocity=1}^{\numberVelocities} \kineticEntropy_{\indexVelocity}(\distributionFunction_{\indexVelocity, \vectorial{\indexSpace}}^{\indexTime + 1, \collided}) - \sum\limits_{\indexVelocity=1}^{\numberVelocities} \kineticEntropy_{\indexVelocity}(\distributionFunction_{\indexVelocity, \vectorial{\indexSpace}}^{\indexTime, \collided})}{\timeStep} \leq \frac{\sum\limits_{\indexVelocity=1}^{\numberVelocities} \kineticEntropy_{\indexVelocity}(\distributionFunction_{\indexVelocity, \vectorial{\indexSpace} - \vectorial{\discreteVelocityLetter}_{\indexVelocity}/\latticeVelocity}^{\indexTime, \collided})- \sum\limits_{\indexVelocity=1}^{\numberVelocities} \kineticEntropy_{\indexVelocity}(\distributionFunction_{\indexVelocity, \vectorial{\indexSpace}}^{\indexTime, \collided})}{\timeStep}.
    \end{equation*}
    Considering a test function $\testFunction\in C_c^1((0, +\infty)\times\reals^{\spatialDimensionality})$ such that $\testFunction\geq 0$ and its discretization, as in the proof of \Cref{thm:convergenceWeakSolution}, yields
    \begin{equation*}
        \timeStep\spaceStep^{\spatialDimensionality}\sum_{\indexTime = 1}^{+\infty} \sum_{\vectorial{\indexSpace}\in\relatives^{\spatialDimensionality}} \sum\limits_{\indexVelocity=1}^{\numberVelocities}\kineticEntropy_{\indexVelocity}(\distributionFunction_{\indexVelocity, \vectorial{\indexSpace}}^{\indexTime, \collided})\frac{\testFunctionDiscrete_{\vectorial{\indexSpace}}^{\indexTime - 1} - \testFunctionDiscrete_{\vectorial{\indexSpace}}^{\indexTime}}{\timeStep} 
        \leq \timeStep\spaceStep^{\spatialDimensionality}\sum_{\indexTime\in\naturals} \sum_{\vectorial{\indexSpace}\in\relatives^{\spatialDimensionality}} \sum\limits_{\indexVelocity=1}^{\numberVelocities}\kineticEntropy_{\indexVelocity}(\distributionFunction_{\indexVelocity, \vectorial{\indexSpace}}^{\indexTime, \collided})\frac{\testFunctionDiscrete_{\vectorial{\indexSpace} + \vectorial{\discreteVelocityLetter}_{\indexVelocity}/\latticeVelocity}^{\indexTime} - \testFunctionDiscrete_{\vectorial{\indexSpace}}^{\indexTime}}{\timeStep}.
    \end{equation*}
    In terms of integrals
    \begin{multline*}
        \int_{\timeStep}^{+\infty}\hspace{-0.5em}\int_{\reals^{\spatialDimensionality}}\sum\limits_{\indexVelocity=1}^{\numberVelocities} \kineticEntropy_{\indexVelocity}(\distributionFunctionsAsFunctionComponent{\indexVelocity}^{\collided}(\timeVariable, \vectorial{\spaceVariable}))\frac{\testFunctionWithStep(\timeVariable - \timeStep, \vectorial{\spaceVariable})-\testFunctionWithStep(\timeVariable, \vectorial{\spaceVariable})}{\timeStep}\differential{\vectorial{\spaceVariable}}\differential{\timeVariable} \\
        \leq\int_{0}^{+\infty}\hspace{-0.5em}\int_{\reals^{\spatialDimensionality}}\sum\limits_{\indexVelocity=1}^{\numberVelocities} \kineticEntropy_{\indexVelocity}(\distributionFunctionsAsFunctionComponent{\indexVelocity}^{\collided}(\timeVariable, \vectorial{\spaceVariable})) \frac{\testFunctionWithStep(\timeVariable, \vectorial{\spaceVariable}+\spaceStep\vectorial{\discreteVelocityLetter}_{\indexVelocity}/\latticeVelocity) - \testFunctionWithStep(\timeVariable, \vectorial{\spaceVariable})}{\timeStep}\differential{\vectorial{\spaceVariable}}\differential{\timeVariable}.
    \end{multline*}
    Formally (rigorous justifications to exchange limits and integrals can be obtained as for \Cref{thm:convergenceWeakSolution}), considering that at the limit, there is no difference between starred and unstarred quantities, the left-hand side tends to 
    \begin{multline*}
        -\int_{0}^{+\infty}\hspace{-0.5em}\int_{\reals^{\spatialDimensionality}}\sum\limits_{\indexVelocity=1}^{\numberVelocities} |\distributionFunctionLetter_{\indexVelocity}^{\atEquilibrium}(\limitConservedMoment(\timeVariable, \vectorial{\spaceVariable})) - \distributionFunctionLetter_{\indexVelocity}^{\atEquilibrium}(\krushkovParameter)|\partial_{\timeVariable}\testFunction(\timeVariable, \vectorial{\spaceVariable})\differential{\vectorial{\spaceVariable}}\differential{\timeVariable}\\
        =-\int_{0}^{+\infty}\hspace{-0.5em}\int_{\reals^{\spatialDimensionality}}|\limitConservedMoment(\timeVariable, \vectorial{\spaceVariable})- \krushkovParameter|\partial_{\timeVariable}\testFunction(\timeVariable, \vectorial{\spaceVariable})\differential{\vectorial{\spaceVariable}}\differential{\timeVariable}
    \end{multline*} 
    by monotonicity of the equilibria, since $\krushkovParameter\in[-\maximumInitialDatum, \maximumInitialDatum]$.
    The right-hand side tends to
    \begin{multline*}
        \int_{0}^{+\infty}\hspace{-0.5em}\int_{\reals^{\spatialDimensionality}} \sum_{\indexVelocity=1}^{\numberVelocities}|\distributionFunctionLetter_{\indexVelocity}^{\atEquilibrium}(\limitConservedMoment(\timeVariable, \vectorial{\spaceVariable}))-\distributionFunctionLetter_{\indexVelocity}^{\atEquilibrium}(\krushkovParameter)|\sum_{\indexDirection = 1}^{\spatialDimensionality} \discreteVelocityLetter_{\indexVelocity, \indexDirection}\partial_{\spaceVariable_{\indexDirection}}\testFunction(\timeVariable, \vectorial{\spaceVariable})\differential{\vectorial{\spaceVariable}}\differential{\timeVariable}\\
        =\int_{0}^{+\infty}\hspace{-0.5em}\int_{\reals^{\spatialDimensionality}} \sum_{\indexVelocity=1}^{\numberVelocities}\sign(\limitConservedMoment(\timeVariable, \vectorial{\spaceVariable})-\krushkovParameter)(\distributionFunctionLetter_{\indexVelocity}^{\atEquilibrium}(\limitConservedMoment(\timeVariable, \vectorial{\spaceVariable}))-\distributionFunctionLetter_{\indexVelocity}^{\atEquilibrium}(\krushkovParameter))\sum_{\indexDirection = 1}^{\spatialDimensionality} \discreteVelocityLetter_{\indexVelocity, \indexDirection}\partial_{\spaceVariable_{\indexDirection}}\testFunction(\timeVariable, \vectorial{\spaceVariable})\differential{\vectorial{\spaceVariable}}\differential{\timeVariable}\\
         = \int_{0}^{+\infty}\hspace{-0.5em}\int_{\reals^{\spatialDimensionality}}\sign(\limitConservedMoment(\timeVariable, \vectorial{\spaceVariable})-\krushkovParameter)\\
         \times\sum_{\indexDirection = 1}^{\spatialDimensionality} \sum_{\indexLink=1}^{\numberLinks} \discreteVelocityLetter_{2\indexLink, \indexDirection} (\distributionFunctionLetter_{2\indexLink}^{\atEquilibrium}(\limitConservedMoment(\timeVariable, \vectorial{\spaceVariable})) - \distributionFunctionLetter_{2\indexLink + 1}^{\atEquilibrium}(\limitConservedMoment(\timeVariable, \vectorial{\spaceVariable})) - \distributionFunctionLetter_{2\indexLink}^{\atEquilibrium}(\krushkovParameter) + \distributionFunctionLetter_{2\indexLink + 1}^{\atEquilibrium}(\krushkovParameter)) \partial_{\spaceVariable_{\indexDirection}}\testFunction(\timeVariable, \vectorial{\spaceVariable})\differential{\vectorial{\spaceVariable}}\differential{\timeVariable} \\
         = \int_{0}^{+\infty}\hspace{-0.5em}\int_{\reals^{\spatialDimensionality}}\sign(\limitConservedMoment(\timeVariable, \vectorial{\spaceVariable})-\krushkovParameter)\sum_{\indexDirection = 1}^{\spatialDimensionality} (\flux_{\indexDirection}(\limitConservedMoment(\timeVariable, \vectorial{\spaceVariable}))-\flux_{\indexDirection}(\krushkovParameter)) \partial_{\spaceVariable_{\indexDirection}}\testFunction(\timeVariable, \vectorial{\spaceVariable})\differential{\vectorial{\spaceVariable}}\differential{\timeVariable},
    \end{multline*}
    where the first equality comes from the monotonicity of the equilibria, and the last one uses \eqref{eq:constraintsConsistency}.
    This is the weak entropy inequality of \eqref{eq:conservationLaw} with Krushkov entropies.
\end{proof}

\section{Numerical experiments}\label{sec:numericalExperiments}

We now corroborate the theoretical findings \emph{via} numerical simulations, conducted using the package \texttt{pyLBM}.\footnote{See \hyperlink{https://pylbm.readthedocs.io}{https://pylbm.readthedocs.io} for more information.}
For the sake of conciseness, we consider test cases for $\spatialDimensionality = 1, 2$. However, the theory that we have previously developed applies to any spatial dimension $\spatialDimensionality \in\naturalsWithoutZero$.

\subsection{\lbmScheme{1}{3} scheme}

\begin{figure}
    \begin{center}
        \includegraphics[width = 1\textwidth]{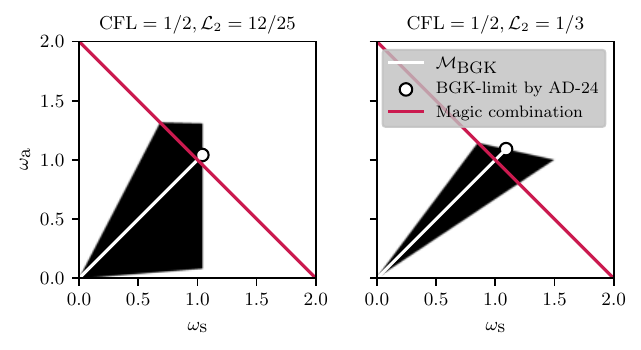}    
    \end{center}\caption{\label{fig:D1Q3}Monotonicity area $\monotoneZoneTwoD$ (in black) in the $\relaxationParameterSymmetric\relaxationParameterAntiSymmetric$-plane for the \lbmScheme{1}{3} scheme.}
\end{figure}

We consider the example proposed in \cite[Remark 3.2]{aregba2024convergence} in the BGK context, where $\spatialDimensionality = 1$, $\numberLinks = 1$, and $\discreteVelocityLetter_2 = \latticeVelocity$. The consistency constraints \eqref{eq:constraintsConsistency3} yield $\equilibriumCoefficientLinear_1 = 1 - 2\equilibriumCoefficientLinear_2$, $\equilibriumCoefficientFlux_1 = 0$, and $\equilibriumCoefficientFlux_2 = \tfrac{1}{2\latticeVelocity}$, leaving $\equilibriumCoefficientLinear_2$ as a free parameter. 
We compare the monotonicity conditions provided by \Cref{prop:monotonicityConditions} with those from the BGK framework, \confer{} \cite{aregba2024convergence}. For the specific case where $\text{CFL} \definitionEquality \tfrac{1}{\latticeVelocity} \max_{\conservedVariableDiscrete \in [-\maximumInitialDatum, \maximumInitialDatum]} | \flux' (\conservedVariableDiscrete) | = \frac{1}{2}$, we examine two different values for $\equilibriumCoefficientLinear_2$. The corresponding monotonicity regions $\monotoneZoneTwoD$ are illustrated in black in \Cref{fig:D1Q3}.
This example highlights that, compared to the BGK case, increasing $\relaxationParameterAntiSymmetric$ allows for a reduction in numerical diffusion while maintaining monotonicity. We focus on the Burgers flux $\flux(\conservedVariable) = {\conservedVariable^2}/{2}$, set $\latticeVelocity = 2$, and use initial data within the interval $[0, 1]$. The numerical simulations are conducted on the domain $[-1, 1]$, equipped with periodic boundary conditions, until final time $\finalTime=\tfrac{1}{4}$.

\subsubsection{Qualitative properties of the solution: less numerical diffusion thanks to TRT}

\begin{figure}
    \begin{center}
        $\equilibriumCoefficientLinear_2 = \frac{12}{25}$\\
        \includegraphics[width = 0.99\textwidth]{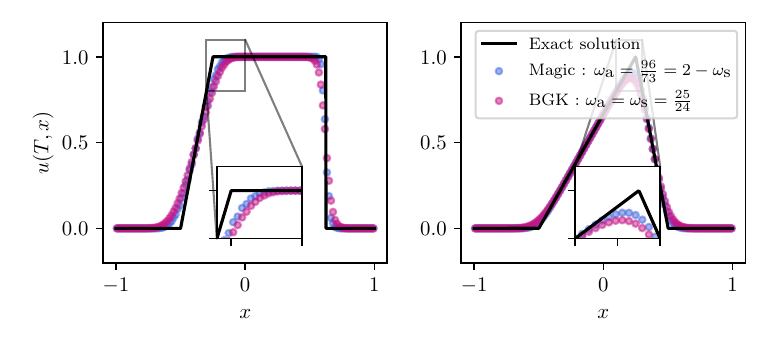} \\
        $\equilibriumCoefficientLinear_2 = \frac{1}{3}$\\
        \includegraphics[width = 0.99\textwidth]{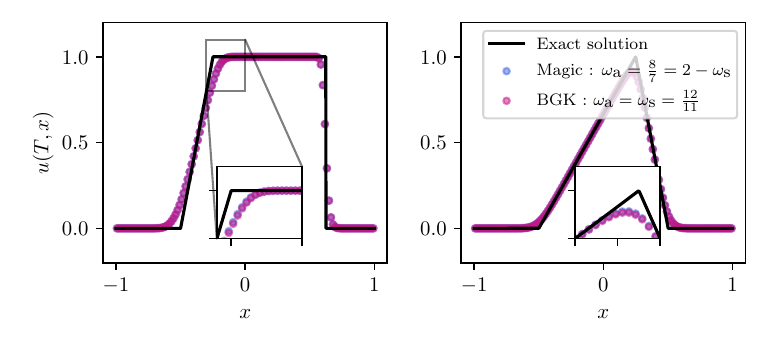}        
    \end{center}\caption{\label{fig:D1Q3plot}Solution of the \lbmScheme{1}{3} schemes and exact solution at final time $1/4$.}
\end{figure}

We conduct tests for $\equilibriumCoefficientLinear_2 = \frac{12}{25}$, examining two cases: the BGK limit $\relaxationParameterSymmetric = \relaxationParameterAntiSymmetric = \frac{25}{24}$, and the so-called ``magic limit'', which maximizes $\relaxationParameterAntiSymmetric$ to minimize numerical diffusion, with $\relaxationParameterSymmetric = \frac{50}{73}$ and $\relaxationParameterAntiSymmetric = \frac{96}{73}$. 
Similarly, for $\equilibriumCoefficientLinear_2 = \frac{1}{3}$, we consider the BGK limit $\relaxationParameterSymmetric = \relaxationParameterAntiSymmetric = \frac{12}{11}$ and the magic limit $\relaxationParameterAntiSymmetric = \frac{8}{7}$.
The results, presented in \Cref{fig:D1Q3plot}, are obtained using a grid with 128 points and two different initial conditions: $\conservedVariable^{\initial}(\spaceVariable) = \indicatorFunction{[0, \frac{1}{2}]}(|\spaceVariable|)$ and $\conservedVariable^{\initial}(\spaceVariable) = (1 - 2|\spaceVariable|)\indicatorFunction{[0, \frac{1}{2}]}(|\spaceVariable|)$. As predicted by the modified equation \eqref{eq:modifiedEquation}, the \strong{TRT scheme exhibits less numerical diffusion}, \confer{} sharper edges, compared to the BGK case.

\subsubsection{Invariant compact set}

\begin{figure}
    \begin{center}
        \includegraphics[width = 0.99\textwidth]{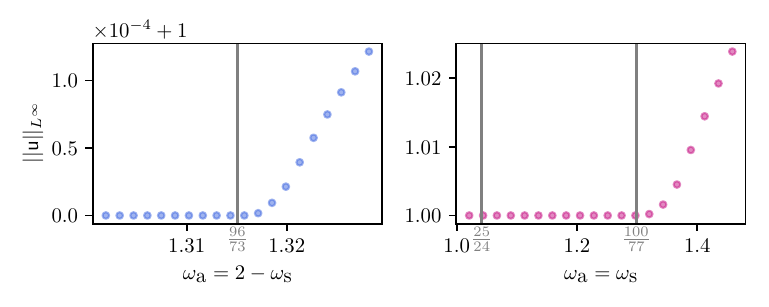}       
    \end{center}\caption{\label{fig:D1Q3_maximum}Time-space maximum of the conserved moment for the \lbmScheme{1}{3} scheme with $\conservedVariable^{\initial}(\spaceVariable) = \indicatorFunction{[0, 1/2]}(|\spaceVariable|)$.}
\end{figure}

Taking $\equilibriumCoefficientLinear_2 = \frac{12}{25}$, we also verify that the conserved moment stays in $[-\maximumInitialDatum, \maximumInitialDatum]$ using the initial condition $\conservedVariable^{\initial}(\spaceVariable) = \indicatorFunction{[0, \frac{1}{2}]}(|\spaceVariable|)$ and a grid with 128 points. 
The results shown in \Cref{fig:D1Q3_maximum} indicate that while the bounds from \Cref{prop:monotonicityConditions} are \strong{necessary in the magic case} to ensure that the maximum principle holds, they are \strong{not necessary}---at least in the considered configuration---for the BGK case ($\monotoneSegmentBGK=(0, \tfrac{25}{24}]$), which is highly diffusive. In fact, in the BGK context, violations of the invariant compact set occur for approximately the same relaxation parameter as in the magic case. In particular, the threshold for the violation of the invariant compact set in the BGK case manifests roughly at $\relaxationParameterSymmetric=\relaxationParameterAntiSymmetric=\tfrac{100}{77}$, which is the one from \eqref{eq:inequalityOmegaBGK} discarding the (first) constraint imposed by the zero velocity.
This behavior highlights a key distinction: in the BGK framework, the inclusion of the zero velocity imposes stricter monotonicity conditions.
However, the specific initialization \eqref{eq:initializationEquilibrium} mitigates the issues caused by violating these conditions, effectively preserving desirable properties even in the presence of such violations.

\subsubsection{Convergence under mesh refinement}

\begin{table}
    \begin{center}\caption{\label{tab:magicD1Q3}Errors and orders of convergence under magic combination ($\relaxationParameterSymmetric + \relaxationParameterAntiSymmetric = 2$).}
        \begin{footnotesize}\setlength{\tabcolsep}{1.5pt}\renewcommand{\arraystretch}{1.2}
            \begin{tabular}{|c|cc|cc|cc|cc|cc|cc|}
                \hline
                $\spaceStep$ & \multicolumn{2}{c|}{$\relaxationParameterAntiSymmetric = 1$} & \multicolumn{2}{c|}{$\relaxationParameterAntiSymmetric = \tfrac{169}{146}$} & \multicolumn{2}{c|}{$\relaxationParameterAntiSymmetric = \tfrac{96}{73}$} & \multicolumn{2}{c|}{$\relaxationParameterAntiSymmetric = \tfrac{3}{2}$} & \multicolumn{2}{c|}{$\relaxationParameterAntiSymmetric = \tfrac{199}{100}$} & \multicolumn{2}{c|}{$\relaxationParameterAntiSymmetric = 2$}\\
                & Error & Ord. & Error & Ord. & Error & Ord. & Error & Ord. & Error & Ord. & Error & Ord. \\
                \hline 
                \hline
                \multicolumn{13}{|c|}{Indicator function initial datum: $\conservedVariable^{\initial}(\spaceVariable) = \indicatorFunction{[0, 1/2]}(|\spaceVariable|)$}\\
                \hline
                3.13E-02 &	1.49E-01 &	     & 1.28E-01 &	     & 1.10E-01 &	     & 9.23E-02 &      & 1.36E-01 &      & 1.41E-01 &	      \\                         
                1.56E-02 &	9.26E-02 &	0.69 & 7.68E-02 &	0.73 & 6.37E-02 &	0.79 & 5.12E-02 & 0.85 & 9.87E-02 & 0.46 & 1.12E-01 &	0.36 \\ 
                7.81E-03 &	5.55E-02 &	0.74 & 4.50E-02 &	0.77 & 3.65E-02 &	0.81 & 2.85E-02 & 0.85 & 9.31E-02 & 0.08 & 1.23E-01 &	-0.13\\ 
                3.91E-03 &	3.22E-02 &	0.78 & 2.59E-02 &	0.80 & 2.07E-02 &	0.82 & 1.59E-02 & 0.85 & 6.99E-02 & 0.42 & 1.14E-01 &	0.10 \\ 
                1.95E-03 &	1.84E-02 &	0.81 & 1.47E-02 &	0.82 & 1.16E-02 &	0.83 & 8.77E-03 & 0.85 & 4.24E-02 & 0.72 & 1.08E-01 &	0.08 \\ 
                9.77E-04 &	1.04E-02 &	0.83 & 8.21E-03 &	0.84 & 6.46E-03 &	0.85 & 4.81E-03 & 0.87 & 2.24E-02 & 0.92 & 1.06E-01 &	0.03 \\ 
                4.88E-04 &	5.80E-03 &	0.84 & 4.55E-03 &	0.85 & 3.56E-03 &	0.86 & 2.62E-03 & 0.88 & 1.12E-02 & 1.00 & 1.02E-01 &	0.06 \\ 
                2.44E-04 &	3.21E-03 &	0.86 & 2.50E-03 &	0.86 & 1.94E-03 &	0.87 & 1.41E-03 & 0.89 & 5.69E-03 & 0.98 & 1.01E-01 &	0.02 \\ 
                1.22E-04 &	1.76E-03 &	0.87 & 1.37E-03 &	0.88 & 1.05E-03 &	0.88 & 7.60E-04 & 0.90 & 2.86E-03 & 0.99 & 9.92E-02 &	0.02 \\ 
                6.10E-05 &	9.57E-04 &	0.88 & 7.39E-04 &	0.89 & 5.67E-04 &	0.89 & 4.06E-04 & 0.90 & 1.43E-03 & 1.00 & 9.87E-02 &	0.01 \\
                \hline
                \multicolumn{13}{|c|}{Hat function initial datum: $ \conservedVariable^{\initial}(\spaceVariable) = (1-2|\spaceVariable|)\indicatorFunction{[0, 1/2]}(|\spaceVariable|)$}\\
                \hline
                3.13E-02 & 5.98E-02 &      & 4.53E-02 &      & 3.38E-02 &      & 2.28E-02	&      & 8.48E-03	&      & 8.66E-03   &       \\                         
                1.56E-02 & 3.12E-02 & 0.94 & 2.32E-02 & 0.97 & 1.70E-02 & 1.00 & 1.12E-02	& 1.03 & 3.19E-03	& 1.41 & 3.28E-03	& 1.40 \\ 
                7.81E-03 & 1.59E-02 & 0.97 & 1.18E-02 & 0.98 & 8.51E-03 & 1.00 & 5.53E-03	& 1.01 & 1.21E-03	& 1.39 & 1.29E-03	& 1.35 \\ 
                3.91E-03 & 8.08E-03 & 0.98 & 5.92E-03 & 0.99 & 4.27E-03 & 1.00 & 2.75E-03	& 1.01 & 4.66E-04	& 1.38 & 5.12E-04	& 1.33 \\ 
                1.95E-03 & 4.07E-03 & 0.99 & 2.98E-03 & 0.99 & 2.14E-03 & 1.00 & 1.38E-03	& 1.00 & 1.80E-04	& 1.38 & 2.05E-04	& 1.32 \\ 
                9.77E-04 & 2.05E-03 & 0.99 & 1.49E-03 & 1.00 & 1.07E-03 & 1.00 & 6.87E-04	& 1.00 & 7.15E-05	& 1.33 & 8.31E-05	& 1.30 \\ 
                4.88E-04 & 1.03E-03 & 1.00 & 7.48E-04 & 1.00 & 5.36E-04 & 1.00 & 3.43E-04	& 1.00 & 2.84E-05	& 1.33 & 3.37E-05	& 1.30 \\ 
                2.44E-04 & 5.14E-04 & 1.00 & 3.74E-04 & 1.00 & 2.68E-04 & 1.00 & 1.72E-04	& 1.00 & 1.13E-05	& 1.33 & 1.37E-05	& 1.30 \\ 
                1.22E-04 & 2.57E-04 & 1.00 & 1.87E-04 & 1.00 & 1.34E-04 & 1.00 & 8.58E-05	& 1.00 & 4.53E-06	& 1.32 & 5.54E-06	& 1.30 \\ 
                6.10E-05 & 1.29E-04 & 1.00 & 9.37E-05 & 1.00 & 6.70E-05 & 1.00 & 4.29E-05	& 1.00 & 1.83E-06	& 1.31 & 2.25E-06	& 1.30 \\ 
                \hline
            \end{tabular}    
        \end{footnotesize}
    \end{center}
\end{table}

\begin{table}
    \begin{center}\caption{\label{tab:BGKD1Q3}Errors and orders of convergence under BGK ($\relaxationParameterSymmetric = \relaxationParameterAntiSymmetric$).}
        \begin{footnotesize}\setlength{\tabcolsep}{2pt}\renewcommand{\arraystretch}{1.2}
            \begin{tabular}{|c|cc|cc|cc|cc|cc|}
                \hline
                $\spaceStep$ &  \multicolumn{2}{c|}{$\relaxationParameterAntiSymmetric = \tfrac{49}{48}$} & \multicolumn{2}{c|}{$\relaxationParameterAntiSymmetric = \tfrac{25}{24}$} & \multicolumn{2}{c|}{$\relaxationParameterAntiSymmetric = \tfrac{3}{2}$} & \multicolumn{2}{c|}{$\relaxationParameterAntiSymmetric = \tfrac{199}{100}$} & \multicolumn{2}{c|}{$\relaxationParameterAntiSymmetric = 2$}\\
                & Error & Ord. & Error & Ord. & Error & Ord. & Error & Ord. & Error & Ord. \\
                \hline 
                \hline
                \multicolumn{11}{|c|}{Indicator function initial datum: $\conservedVariable^{\initial}(\spaceVariable) = \indicatorFunction{[0, 1/2]}(|\spaceVariable|)$}\\
                \hline
                3.12E-02	& 1.46E-01	&      & 1.43E-01&	     & 8.69E-02 &	     & 	1.42E-01 &	     & 1.50E-01 &	     \\     
                1.56E-02	& 9.03E-02	& 0.69 & 8.81E-02&	0.69 & 5.01E-02 &	0.79 & 	1.09E-01 &	0.38 & 1.23E-01 &	0.29 \\
                7.81E-03	& 5.40E-02	& 0.74 & 5.25E-02&	0.75 & 2.84E-02 &	0.82 & 	8.33E-02 &	0.39 & 1.11E-01 &	0.15 \\
                3.91E-03	& 3.13E-02	& 0.79 & 3.04E-02&	0.79 & 1.59E-02 &	0.84 & 	6.57E-02 &	0.34 & 1.00E-01 &	0.14 \\
                1.95E-03	& 1.78E-02	& 0.81 & 1.73E-02&	0.81 & 8.80E-03 &	0.85 & 	4.19E-02 &	0.65 & 9.55E-02 &	0.07 \\
                9.77E-04	& 1.01E-02	& 0.83 & 9.76E-03&	0.83 & 4.83E-03 &	0.87 & 	2.20E-02 &	0.93 & 9.11E-02 &	0.07 \\ 
                4.88E-04	& 5.62E-03	& 0.84 & 5.44E-03&	0.84 & 2.63E-03 &	0.88 & 	1.11E-02 &	0.99 & 8.94E-02 &	0.03 \\
                2.44E-04	& 3.10E-03	& 0.86 & 3.00E-03&	0.86 & 1.42E-03 &	0.89 & 	5.53E-03 &	1.00 & 1.76E+13 &	-47.5\\ 
                1.22E-04	& 1.70E-03	& 0.87 & 1.64E-03&	0.87 & 7.62E-04 &	0.90 & 	2.76E-03 &	1.00 & 8.80E+12 &	1.00 \\ 
                6.10E-05	& 9.25E-04	& 0.88 & 8.94E-04&	0.88 & 4.07E-04 &	0.90 & 	1.38E-03 &	1.00 & 4.40E+12 &	1.00 \\ 
                \hline
                \multicolumn{11}{|c|}{Hat function initial datum: $ \conservedVariable^{\initial}(\spaceVariable) = (1-2|\spaceVariable|)\indicatorFunction{[0, 1/2]}(|\spaceVariable|)$}\\
                \hline 
                3.12E-02	& 5.77E-02	&      & 5.56E-02 &	     & 2.28E-02&	     & 8.34E-03 &	     & 8.51E-03 &	    \\     
                1.56E-02	& 3.00E-02	& 0.94 & 2.89E-02 &	0.95 & 1.12E-02&	1.03 & 3.08E-03 &	1.44 & 3.22E-03 &	1.40\\
                7.81E-03	& 1.53E-02	& 0.97 & 1.47E-02 &	0.97 & 5.53E-03&	1.01 & 1.17E-03 &	1.39 & 1.27E-03 &	1.35\\
                3.91E-03	& 7.76E-03	& 0.98 & 7.45E-03 &	0.98 & 2.75E-03&	1.01 & 4.48E-04 &	1.39 & 5.07E-04 &	1.32\\
                1.95E-03	& 3.91E-03	& 0.99 & 3.75E-03 &	0.99 & 1.37E-03&	1.00 & 1.70E-04 &	1.39 & 2.02E-04 &	1.33\\
                9.77E-04	& 1.96E-03	& 0.99 & 1.88E-03 &	0.99 & 6.87E-04&	1.00 & 6.75E-05 &	1.34 & 8.18E-05 &	1.31\\
                4.88E-04	& 9.85E-04	& 1.00 & 9.44E-04 &	1.00 & 3.43E-04&	1.00 & 2.68E-05 &	1.33 & 3.30E-05 &	1.31\\
                2.44E-04	& 4.93E-04	& 1.00 & 4.73E-04 &	1.00 & 1.72E-04&	1.00 & 1.07E-05 &	1.32 & 1.34E-05 &	1.30\\
                1.22E-04	& 2.47E-04	& 1.00 & 2.37E-04 &	1.00 & 8.58E-05&	1.00 & 4.30E-06 &	1.32 & 5.41E-06 &	1.30\\
                6.10E-05	& 1.23E-04	& 1.00 & 1.18E-04 &	1.00 & 4.29E-05&	1.00 & 1.74E-06 &	1.30 & 2.18E-06 &	1.31\\
                \hline
            \end{tabular}    
        \end{footnotesize}
    \end{center}
\end{table}

For the magic combination, regardless of being in the monotonicity zone, we obtain the results presented in \Cref{tab:magicD1Q3}, where the $\lebesgueTime{\infty}L_{\spaceVariable}^1$-error on $\conservedVariableDiscrete$ is provided. 
To interpret these results, recall that if the solution belongs to the Besov space $\besovSpace{1}{s}{\infty}$ with $0 < s < \mu + 1$ (where $\mu$ is the order of the scheme), one typically expects convergence in the $L_{\spaceVariable}^1$-norm at order $\bigO{\spaceStep^{\frac{s\mu}{\mu + 1}}}$, assuming the scheme converges and in the linear case, as shown in \cite[Theorem 4.2]{brenner2006besov}. 
For the indicator function case, we observe that for $\relaxationParameterAntiSymmetric \leq \tfrac{3}{2}$, the empirical convergence rate exceeds $1/2$. This aligns with the fact that in non-linear settings (aside from pathological initial data, \confer{} \cite{sabac1997optimal}), shock solutions often yield better convergence rates than expected. When $\relaxationParameterAntiSymmetric$ approaches, but does not equal, two, the scheme behaves as a second-order accurate one ($\mu = 2$). Since $\conservedVariable^{\initial} \in \besovSpace{1}{1}{\infty}$ \cite[Section 2.4, Example II]{brenner2006besov}, the linear theory predicts an order $2/3$. However, in our non-linear setting, the empirical rate is slightly better, equating to one. Conversely, when $\relaxationParameterAntiSymmetric = 2$, the scheme becomes unstable and fails to converge.
For the hat function case, where $\conservedVariable^{\initial} \in \besovSpace{1}{2}{\infty}$, we consistently observe first-order empirical convergence for $\relaxationParameterAntiSymmetric$ way apart from $2$. This matches the linear theoretical prediction. As $\relaxationParameterAntiSymmetric$ approaches or equals two, the scheme demonstrates second-order accuracy, with an empirical convergence rate close to $4/3 = 1.\overline{3}$, again consistent with the linear theory.

The results for the BGK case, provided in \Cref{tab:BGKD1Q3}, are analogous.

\subsubsection{Convergence to the equilibrium}\label{sec:D1Q3convergenceEquilibrium}

\begin{figure}
    \begin{center}
        \includegraphics[width=0.99\textwidth]{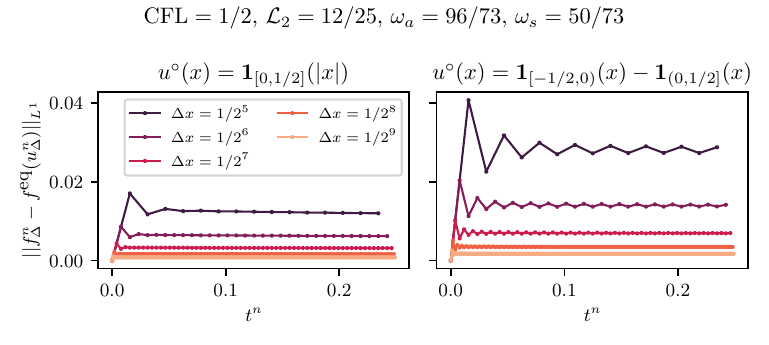}
        \includegraphics[width=0.99\textwidth]{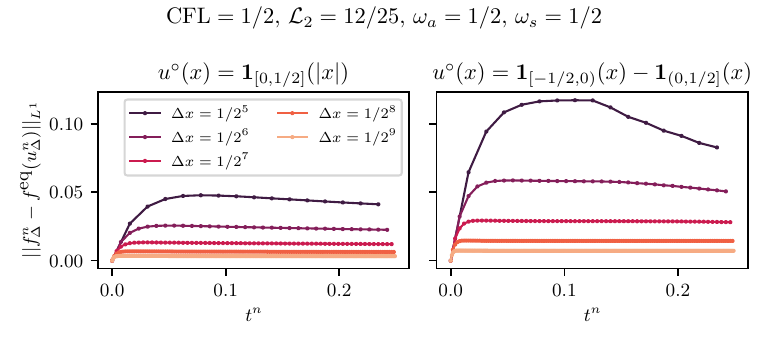}
    \end{center}\caption{\label{fig:D1Q3convergenceEquilibrium}Distance from the equilibrium for the \lbmScheme{1}{3} scheme at different resolutions, parameters, and initial data.}
\end{figure}

We monitor the quantity $\lVert \distributionFunctionsAsFunctionSpace{\indexTime} - \vectorial{\distributionFunctionLetter}^{\atEquilibrium}(\conservedVariableDiscreteAsFunctionSpace{\indexTime})\rVert_{\lebesgueSpace{1}}$, testing for various parameters, resolutions, and initial data, as illustrated in \Cref{fig:D1Q3convergenceEquilibrium}. 
In the left plot, the initial datum has $\totalVariation{\conservedVariable^{\initial}} = 2$, while in the right plot, the total variation is doubled: $\totalVariation{\conservedVariable^{\initial}} = 4$. The results confirm that \eqref{eq:convergenceEquilibrium} and \eqref{eq:geometricConvergenceEquilibrium} are \strong{sharp}: the trend of the observed $\lVert \distributionFunctionsAsFunctionSpace{\indexTime} - \vectorial{\distributionFunctionLetter}^{\atEquilibrium}(\conservedVariableDiscreteAsFunctionSpace{\indexTime})\rVert_{L^{1}}$ scales linearly with both $\spaceStep$ and $\totalVariation{\conservedVariable^{\initial}}$. 
Finally, the observed convergence to the stationary value is geometrically damped. When relaxation parameters exceed one, oscillations appear in the convergence behavior.

\subsection{\lbmScheme{2}{5} scheme}

We take $\spatialDimensionality=2$ with $\numberLinks = 2$, and $\vectorial{\discreteVelocityLetter}_2 = \transpose{(\latticeVelocity, 0)}$, $\vectorial{\discreteVelocityLetter}_4 = \transpose{(0, \latticeVelocity)}$.
After enforcing \eqref{eq:constraintsConsistency3}, $\equilibriumCoefficientLinear_2$ and $\equilibriumCoefficientLinear_4$ remain free.
    We consider the setting where we take the flux $\flux(\conservedVariable)$ of a 1D conservation law and we construct 2D fluxes by $\flux_1 = \cos(\theta)\flux$ and $\flux_2 = \sin(\theta)\flux$: we obtain, for example, $\monotoneZoneTwoD$ as in \Cref{fig:D2Q5}.
    We take $\theta = \pi/4$ with $\flux(\conservedVariable)=u^2/2$.

    \begin{figure}
        \begin{center}
            \includegraphics[width = 1\textwidth]{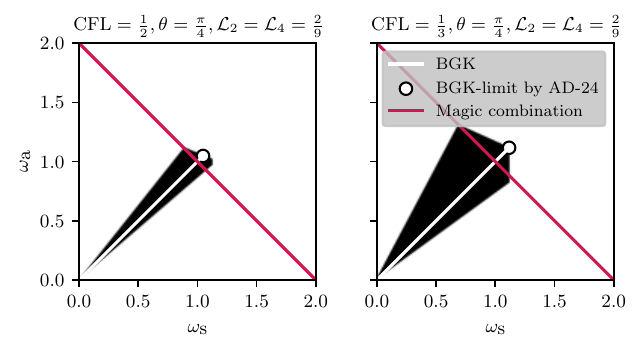}    
        \end{center}\caption{\label{fig:D2Q5}Monotonicity area (in black) in the $\relaxationParameterSymmetric\relaxationParameterAntiSymmetric$-plane for the \lbmScheme{2}{5} scheme.}
    \end{figure}

\subsubsection{Qualitative properties and invariant compact set}

    \begin{figure}
        \begin{center}
            \includegraphics[width=0.99\textwidth]{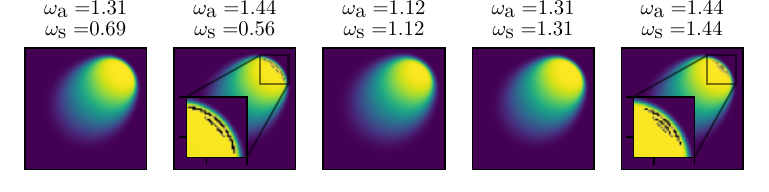}
        \end{center}\caption{\label{fig:D2Q5result}Solution of the \lbmScheme{2}{5} scheme at final time $1/4$, using 64 grid-points \emph{per} direction. The black areas indicate points where the invariant compact set $[-\maximumInitialDatum, \maximumInitialDatum]$ for $\conservedVariableDiscrete$ is violated.}
    \end{figure}

We simulate on the bounded domain $[-1, 1]^2$ endowed with periodic boundary conditions with the parameters on the right of \Cref{fig:D2Q5}, using the initial datum $\conservedVariable^{\initial}(\vectorial{\spaceVariable}) = \indicatorFunction{[0, 1]}(\sqrt{\spaceVariable_1^2 + \spaceVariable_2^2})$.
The results in \Cref{fig:D2Q5result} are given both for the magic and BGK setting.
Similarly to the one-dimensional case, they highlight that violating the monotonicity constraint for the magic setting immediately creates data outside the invariant compact set.
On the other hand, the monotonicity constraints in the BGK case are stronger than what is needed just to preserve the invariant compact set for $\conservedVariableDiscrete$.
    
\section{Conclusions and perspectives}\label{sec:Conclusions}

In this work, we analyzed \strong{monotonicity} for a broad class of lattice Boltzmann schemes, focusing on the \strong{two-relaxation-times} (TRT) model, which fostered a proof of \strong{convergence} in the nonlinear case towards the entropy solution of a conservation law.
By introducing a second relaxation parameter, we demonstrated how \strong{numerical diffusion can be reduced}---leading to improved accuracy---while maintaining the monotonicity required to establish convergence in the spirit of the Lax-Wendroff theorem. Notably, we showed that the monotonicity conditions derived for the so-called ``magic'' case are \strong{optimal} and cannot be further improved.
Finally, let us mention that it was out of the scope of the paper to provide a comprehensive study of the best choice of the (many) free parameters that lattice Boltzmann schemes feature, in terms of the tradeoff between stability with respect to a given norm and accuracy.
Looking ahead, an important challenge lies in recovering desirable properties within the BGK framework, even when some coefficients in the relaxation phase are \strong{negative}. A promising direction could involve approaches explored in \cite{hundsdorfer2003monotonicity, hundsdorfer2006monotonicity} and \cite{bellotti2023numerical}.
Another interesting path of research tackles the study of the convergence for schemes where equilibria feature \strong{non-linearities} which are not only proportional to the flux of the PDE.
Finally, extending these results to \strong{systems} of conservation laws---a significantly more complex setting---remains a compelling avenue for future research.

\section*{Acknowledgments}
The authors thank the two anonymous reviewers for their useful comments, which helped improving this work.

\bibliographystyle{siamplain}
\bibliography{biblio}

\begin{thebibliography}{10}

\bibitem{aregba2024convergence}
{\sc D.~Aregba-Driollet}, {\em {Convergence of Lattice Boltzmann methods with
  overrelaxation for a nonlinear conservation law}}, ESAIM: Mathematical
  Modelling and Numerical Analysis, 58 (2024), pp.~1935--1958.

\bibitem{aregba2000discrete}
{\sc D.~Aregba-Driollet and R.~Natalini}, {\em Discrete kinetic schemes for
  multidimensional systems of conservation laws}, SIAM Journal on Numerical
  Analysis, 37 (2000), pp.~1973--2004.

\bibitem{bellotti2023monotonicity}
{\sc T.~Bellotti}, {\em {Monotonicity for genuinely multi-step methods: Results
  and issues from a simple lattice Boltzmann scheme}}, in International
  Conference on Finite Volumes for Complex Applications, Springer, 2023,
  pp.~33--41.

\bibitem{bellotti2023numerical}
{\sc T.~Bellotti}, {\em Numerical analysis of lattice Boltzmann schemes: from
  fundamental issues to efficient and accurate adaptive methods}, PhD thesis,
  Institut Polytechnique de Paris, 2023.

\bibitem{bellotti2023truncation}
{\sc T.~Bellotti}, {\em {Truncation errors and modified equations for the
  lattice Boltzmann method via the corresponding Finite Difference schemes}},
  ESAIM: Mathematical Modelling and Numerical Analysis, 57 (2023),
  pp.~1225--1255.

\bibitem{bellotti2024initialisation}
{\sc T.~Bellotti}, {\em {Initialisation from lattice Boltzmann to multi-step
  Finite Difference methods: modified equations and discrete observability}},
  Journal of Computational Physics, 504 (2024), p.~112871.

\bibitem{bellotti2022finite}
{\sc T.~Bellotti, B.~Graille, and M.~Massot}, {\em {Finite Difference
  formulation of any lattice Boltzmann scheme}}, Numerische Mathematik, 152
  (2022), pp.~1--40.

\bibitem{bouchut1999construction}
{\sc F.~Bouchut}, {\em {Construction of BGK models with a family of kinetic
  entropies for a given system of conservation laws}}, Journal of Statistical
  Physics, 95 (1999), pp.~113--170.

\bibitem{brenner2006besov}
{\sc P.~Brenner, V.~Thom{\'e}e, and L.~B. Wahlbin}, {\em Besov spaces and
  applications to difference methods for initial value problems}, vol.~434,
  Springer, 2006.

\bibitem{caetano2024result}
{\sc F.~Caetano, F.~Dubois, and B.~Graille}, {\em A result of convergence for a
  mono-dimensional two-velocities lattice boltzmann scheme}, Discrete and
  Continuous Dynamical Systems - S, 17 (2024), pp.~3129--3154,
  \url{https://doi.org/10.3934/dcdss.2023072}.

\bibitem{chen2024unified}
{\sc Y.~Chen, Z.~Chai, and B.~Shi}, {\em {A unified fourth-order
  Bhatnagar-Gross-Krook lattice Boltzmann model for high-dimensional linear
  hyperbolic equations}}, arXiv preprint arXiv:2410.13165,  (2024).

\bibitem{crandall1980monotone}
{\sc M.~G. Crandall and A.~Majda}, {\em Monotone difference approximations for
  scalar conservation laws}, Mathematics of Computation, 34 (1980), pp.~1--21.

\bibitem{dellacherie2014construction}
{\sc S.~Dellacherie}, {\em {Construction and analysis of lattice Boltzmann
  methods applied to a 1D convection-diffusion equation}}, Acta Applicandae
  Mathematicae, 131 (2014), pp.~69--140.

\bibitem{dellar2023magic}
{\sc P.~J. Dellar}, {\em {A magic two-relaxation-time lattice Boltzmann
  algorithm for magnetohydrodynamics}}, Discrete and Continuous Dynamical
  Systems-S, 17 (2024), pp.~3155--3173.

\bibitem{d1992generalized}
{\sc D.~d'Humières}, {\em {Generalized lattice-Boltzmann equations}}, Rarefied
  gas dynamics,  (1992).

\bibitem{dubois2022nonlinear}
{\sc F.~Dubois}, {\em {Nonlinear fourth order Taylor expansion of lattice
  Boltzmann schemes}}, Asymptotic Analysis, 127 (2022), pp.~297--337.

\bibitem{dubois2020notion}
{\sc F.~Dubois, B.~Graille, and S.~R. Rao}, {\em A notion of non-negativity
  preserving relaxation for a mono-dimensional three velocities scheme with
  relative velocity}, Journal of Computational Science, 47 (2020), p.~101181.

\bibitem{ginzburg2006variably}
{\sc I.~Ginzburg}, {\em {Variably saturated flow described with the anisotropic
  lattice Boltzmann methods}}, Computers \& Fluids, 35 (2006), pp.~831--848.

\bibitem{ginzburg:tel-02591565}
{\sc I.~Ginzburg}, {\em {Une variation sur les propri{\'e}t{\'e}s magiques de
  mod{\`e}les de Boltzmann pour l'{\'e}coulement microscopique et
  macroscopique}}, {Th{\`e}se d'Habilitation {\`a} diriger des recherches,
  Universit{\'e} Pierre et Marie Curie Paris}, 2009.

\bibitem{ginzburg2008two}
{\sc I.~Ginzburg, F.~Verhaeghe, and D.~d’Humières}, {\em
  {Two-relaxation-time lattice Boltzmann scheme: About parametrization,
  velocity, pressure and mixed boundary conditions}}, Communications in
  Computational Physics, 3 (2008), pp.~427--478.

\bibitem{godlewski1991hyperbolic}
{\sc E.~Godlewski and P.-A. Raviart}, {\em Hyperbolic systems of conservation
  laws}, no.~3-4, Ellipses, 1991.

\bibitem{graille2014approximation}
{\sc B.~Graille}, {\em {Approximation of mono-dimensional hyperbolic systems: A
  lattice Boltzmann scheme as a relaxation method}}, Journal of Computational
  Physics, 266 (2014), pp.~74--88.

\bibitem{hundsdorfer2006monotonicity}
{\sc W.~Hundsdorfer and S.~Ruuth}, {\em On monotonicity and boundedness
  properties of linear multistep methods}, Mathematics of Computation, 75
  (2006), pp.~655--672.

\bibitem{hundsdorfer2003monotonicity}
{\sc W.~Hundsdorfer, S.~J. Ruuth, and R.~J. Spiteri}, {\em
  Monotonicity-preserving linear multistep methods}, SIAM Journal on Numerical
  Analysis, 41 (2003), pp.~605--623.

\bibitem{natalini1998discrete}
{\sc R.~Natalini}, {\em A discrete kinetic approximation of entropy solutions
  to multidimensional scalar conservation laws}, Journal of Differential
  Equations, 148 (1998), pp.~292--317.

\bibitem{sabac1997optimal}
{\sc F.~Sabac}, {\em The optimal convergence rate of monotone finite difference
  methods for hyperbolic conservation laws}, SIAM Journal on Numerical
  Analysis, 34 (1997), pp.~2306--2318.

\bibitem{teng2010error}
{\sc Z.-H. Teng}, {\em Error bound between monotone difference schemes and
  their modified equations}, Mathematics of Computation, 79 (2010),
  pp.~1473--1491.

\bibitem{warming1974modified}
{\sc R.~F. Warming and B.~J. Hyett}, {\em The modified equation approach to the
  stability and accuracy analysis of finite-difference methods}, Journal of
  Computational Physics, 14 (1974), pp.~159--179.

\end{thebibliography}

\end{document}